\documentclass[10pt]{article}
\usepackage[margin=1in]{geometry}
\usepackage{float}
  \usepackage{amsmath,amssymb,amsthm}
  \usepackage{graphicx,mathtools}
  \usepackage{xcolor,soul}
\usepackage{enumerate}
\usepackage{verbatim}
\usepackage{marginnote}
\usepackage{array}
\usepackage{booktabs}
\usepackage[colorlinks=true]{hyperref}
\hypersetup{urlcolor=blue, citecolor=red, linkcolor=red}

\newcolumntype{L}[1]{>{\raggedright\let\newline\\\arraybackslash\hspace{0pt}}m{#1}}

\numberwithin{equation}{section}

\newtheorem{theorem}{Theorem}[section]
\newtheorem{cor}[theorem]{Corollary}
\newtheorem{lemma}[theorem]{Lemma}

\newtheorem{prop}[theorem]{Proposition} \theoremstyle{definition}
\newtheorem{definition}[theorem]{Definition}
\newtheorem{rem}[theorem]{Remark}

\theoremstyle{plain}
\newtheorem{assumption}{Assumption}

\newcommand{\bra}[1]{\left[#1\right]} 
\newcommand{\mat}[1]{\begin{matrix}#1\end{matrix}} 
\newcommand{\bmat}[1]{\bra{\mat{#1}}} 

\def\beq{\begin{equation} } \def\eeq{\end{equation}}

\def\supp{\hbox{supp}}

\usepackage{mathtools}

\DeclarePairedDelimiterX\set[2]{\{}{\}}{#1\,\delimsize\vert\,#2}

\setstcolor{red}

\def\RR{\mathbb R}

\def\ben{\begin{enumerate} }

\def\een{\end{enumerate} }

\def \R{ {\mathbb R}}
  
\def \p { \partial}

  \def \g{ \mathbf{g}}

\def \Id{ \text{Id} } 
\def \div{ \text{div}}  
\def \C{ \mathbf{C}}

\def \Ups{ \RR^3}

\def \WF{ \text{WF}} 

\def \BtoB{_{\p \to \p}}
\def \CtoS{_{\mathbf C \to \mathbf S}}
\def \CtoB{_{\mathbf C \to \p}}
\def \BtoS{_{\p \to \mathbf S}}
\def \Op{ \mathcal{P}} 
\def \Z{ \mathbf{Z}}
\def \stau{ \sqrt \tau}

\def \sig{\kappa}
\def \q{q}

\renewcommand{\restriction}{\mathord{\upharpoonright}}

\def\g{f} 

\newcommand\PS{{P\!/\!S}}

\makeatletter
\def\blfootnote{\gdef\@thefnmark{}\@footnotetext}
\makeatother

\title{Recovery of piecewise smooth density and Lam\'{e} parameters from high-frequency exterior Cauchy data}

\author{Sombuddha Bhattacharyya
\thanks{Department of Mathematics, Indian Institute of Science Education and Research Bhopal, India.}
\and
Maarten V. de Hoop
\thanks{Simons Chair in Computational and Applied Mathematics and Earth Science, Rice University, Houston TX, USA. \texttt{mdehoop@rice.edu}}
\and
Vitaly Katsnelson
\thanks{College of Arts and Sciences, New York Institute of Technology, New York NY, USA. \texttt{vkatsnel@nyit.edu}. }
\and
Gunther Uhlmann
\thanks{Department of Mathematics, University of Washington, Seattle WA, USA. \texttt{gunther@math.washington.edu}}
\thanks{Institute for Advanced Study, The Hong Kong University of Science and Technology, Clear Water Bay, Hong Kong.}
}

\begin{document}
\maketitle

\begin{abstract}

We consider an isotropic elastic medium occupying a bounded domain $\Omega \subset \RR^3$ whose density and Lam\'{e} parameters are piecewise smooth.
 In the elastic wave initial value inverse problem, we are given the solution operator for the elastic wave equation, but only outside $\Omega$ and only for initial data supported outside $\Omega$, and we study the recovery of the density and Lam\'{e} parameters. For known density, results have recently been obtained using the scattering control method to recover wave speeds. Here, we extend this result to include the recovery of the density in addition to the Lam\'{e} parameters under certain geometric conditions using techniques from microlocal analysis and a connection to local tensor tomography.

 \end{abstract}

\subsection*{Acknowledgements
} M.V.d.H. gratefully acknowledges support from the Simons Foundation under the MATH + X program, the National Science Foundation under grant
DMS-1815143, and the corporate members of the Geo-Mathematical Imaging Group at Rice University. G.U. was partly supported by NSF, a Walker Family Endowed Professorship at UW and a Si-Yuan Professorship at IAS, HKUST. S.B. was partly supported by Project no.: 16305018 of the Hong Kong Research Grant Council.

\section{Introduction}
The main goal of this work is to recover a piecewise smooth density of mass in addition to the other elastic parameters in an isotropic elastic setting using exterior measurements.
In general, the wave inverse problem asks for the unknown coefficient(s), representing wave speeds, of a wave equation inside a domain of interest $\Omega$, given knowledge about the equation's solutions (typically on $\p \Omega$). Traditionally, the coefficients are smooth, and the data is the Dirichlet-to-Neumann (DN) map, or its inverse. The main questions are uniqueness and stability: Can the coefficients be recovered from the Dirichlet-to-Neumann map, and is this reconstruction stable relative to perturbations in the data? In the case of a scalar wave equation with smooth coefficients, a number of results by Belishev, Stefanov, Uhlmann, and Vasy \cite{Belishev-multidimenIP,UVLocalRay, SURigidity} have answered the question in the affirmative. For the piecewise smooth case, a novel scattering control method was developed in \cite{CHKUControl} in order to show in \cite{CHKUUniqueness} that uniqueness holds as well for piecewise smooth wave speeds with conormal singularities, under mild geometric conditions. Less is known in the elastic setting as will be described, but several works such as \cite{stefanov2021solidfluid,zhang2020rayleigh,SUV2019transmission} show how to construct an FIO representation of the solution to the elastic wave equation near an interface, which is useful for inverse problems in the hyperbolic elastic setting where coefficients have conormal singularities.
An additional challenge of recovering a coefficient that is not in the principal symbol of the operator is that one needs to solve a tensor tomography problem at some stage of the argument, which has a gauge freedom that obstructs uniqueness \cite{Oksanen2020}. As will be explained, the gauge freedom in our case actually gets exploited to derive an elliptic equation that allows a unique recovery of the density.

 In this work, we recover a piecewise smooth density of mass (in the isotropic elastic wave equation) in the interior of the domain, thereby recovering all three Lam\'{e} parameters. We essentially want to ``image'' the density using high frequency waves. 
    In \cite{SUV2019transmission}, Stefanov, Uhlmann, and Vasy recover piecewise smooth wavespeeds from the Dirichlet-to-Neumann map under certain geometric conditions. 
  Their argument relies on the principal symbol of the elastic operator and the parameterix. As noted in \cite[Remark 10.2]{SUV2019transmission}, their argument does not address unique determination of the density of mass past the first interface, nor at the interface itself. Since the density appears in the lower order part of the elastic operator, it is natural to look at the lower order symbols in the asymptotic expansion of elastic wave solutions. This leads to the inverse problem considered here where we study the lower order constituents of a parameterix, as done by Rachele in \cite{RachDensity} in the smooth case, to recover a piecewise smooth density of mass.

Consider the isotropic elastic wave equation in a bounded domain $\Omega \subset \RR^3$ with smooth boundary. The wave operator for elastodynamics is given as 
\begin{equation}
    \Op=\rho\p^2_t - L
\end{equation}
with
\[
L = \nabla \cdot( \lambda \div \otimes \Id + 2\mu \widehat{\nabla}),
\]
where $\rho$ is the density of mass that we term simply as \emph{density}, $\lambda$ and $\mu$ are the Lam\'{e} parameters, and $\widehat{\nabla}$ is the \emph{symmetric gradient} used to define the strain tensor for an elastic system via $\widehat{\nabla} u = (\nabla u + (\nabla u)^T)/2$ for a vector valued function $u$. Operator $\Op$ acts on a vector-valued distribution $u(x,t)= (u_1,u_2,u_3)$, the \emph{displacement} of the elastic object. We will assume that all three material parameters are piecewise smooth, that is, $C^\infty$ except on a finite set of smooth hypersurfaces (that we describe later) in $\Omega$ with possible jumps there. For the isotropic, elastic setting with smooth parameters, the uniqueness question was settled by Rachele in \cite{Rach00} and Hansen and Uhlmann \cite{HUPolarization}.
 In \cite{CHKUElastic}, the authors extended these results to the isotropic elastic system, where the parameters are piecewise smooth. The main difficulty here is lack of the sharp form of the unique continuation result of Tataru since one has to deal with two different wave speeds.
 The main result of \cite{CHKUElastic} is that under certain geometric assumptions, one can uniquely determination the $\PS$-waveseeds that contain singularities via microlocal analysis, scattering control, and a layer stripping argument akin to \cite{SUVRigidity}. In \cite{CHKUElastic}, it is assumed that the density $\rho$ was trivial, but with similar yet more sophisticated arguments, it was proved in \cite{SUV2019transmission} that this assumption can be dispensed with, and that both piecewise smooth wavespeeds can be recovered from exterior measurements even when the density is piecewise smooth. The simpler case of piecewise analytic and piecewise constant coefficients is considered in \cite{ZhaiPiecewiseAnalytic, OksanenPiecewise2019} and the arguments are quite different than our approach here. 
 In our approach, low frequencies are not required in the data.

Recovering the material density does not simply follow from the arguments in \cite{SUV2019transmission,CHKUElastic}. This is because those arguments only rely on the principal symbol of the elastic operator and the principal term in the high frequency asymptotic expansion of solutions to the elastic wave equation to recover travel times, and the principal symbol contains no information about the density \cite{Rach00,RachDensity}. Rachele also showed in \cite{RachDensity} that the polarization of the waves does not contain information about the density. By looking at the lower order terms of the amplitude of an FIO representation of an elastic wave solution, Rachele shows \cite{RachDensity} that in the smooth setting, one may recover the density as well under certain conditions. That was a global result, but by using the results of \cite{SUVRigidity} in local tensor tomography, Sombuddha shows in \cite{B_density} that one can locally recover the density as well near the boundary in the smooth setting. We aim to extend the results in \cite{RachDensity} to the piecewise smooth setting and recover a piecewise smooth density of mass. We also note that Rachele's results in \cite{Rach00,B_density}
assume that the manifold is simple so that there are no caustics. Here, we do not assume that the manifold is simple, but we do assume a ``convex foliation condition'' that will be described shortly, which allows for caustics, and in particular, is satisfied by manifolds with non-negative curvature. Another novelty is that previous papers \cite{RachBoundary,RachDensity}  used plane waves to recover parameters and lower order amplitude data for tensor tomography, where one can just recover each term in the asymptotic expansion of an FIO representation of the solution. This approach is not known to work to recover such data past an interface. Thus, we use more general distributions and show how to recover ``lower order polarizations'' of elastic waves locally beyond an interface or multiple interfaces.

  In Theorem \ref{Main_Th_1} we show that under certain geometric assumptions and outside a specific set, we can uniquely determine all three Lam\'{e} parameters that contain singularities. We first recover the wave speeds as in \cite{SUV2019transmission,CHKUElastic} via microlocal analysis, scattering control, and a layer stripping argument akin to \cite{SUVRigidity} by recovering and manipulating local travel time data.
   To recover $\rho$ from the boundary up until the first interface, the results of \cite{B_density} may be employed as shown in \cite{SUV2019transmission} where they recover all three parameters up until the first interface, but do not address recovering the density of mass past the first interface.
   To recover the density (and all its derivatives) across the interface, we first send in high frequency waves at the interface and measure their reflected amplitudes (we can do this since we have already recovered the full elastic operator above this interface) to recover the reflection operator at the interface as in \cite{CHKUElastic}. The reflection operator is a classical $0$'th order PsiDO, and we recover all the terms in the polyhomogeneous expansion of its symbol. We may then use the result in
    \cite{BHKU_1} to recover the jet of $\rho$ and all its derivatives infinitesimally below the interface. We then send in high frequency waves that generate a transmitted $P$-wave directly below the interface, which then travels near the interface before returning to it. We measure the amplitude of this $P$-wave to obtain ``lower order polarization data'' which involves the density, and recovering the density gets reduced to a local 2-tensor tomography problem by a careful analysis of this lower order amplitude.  Normally, this would create gauge freedom in the recovery of the density, but we employ the argument of \cite{B_density,RachDensity} to show that the gauge freedom actually leads to an elliptic equation for the density, where we can then locally recover the density in the interior below the interface. We proceed in this way until the density is recovered in all layers.

    A subtle technical point is that we probe the medium with distributions that have wavefront at a single covector (modulo the $\RR^+$ group action) in $T^*\Omega$, and hence, locally and away from any interface, the waves we generate have a wavefront set consisting of a single bicharacteristic. Such distributions do not fall into a specific H\"{o}rmander class of conormal distributions, and our given measurements have the form $FV$ where $F$ is a specific Fourier integral operator with a polyhomogeneous symbol and $V$ is just an arbitrary distribution. To obtain the ``lower order polarization'' amplitudes, we need to extract the principal symbol as well as the lower order symbols associated to $F$ from such data. In Section \ref{s: Weinstein calculus}, we derive a variation of the Weinstein symbol calculus in \cite{weinstein1976} that applies to arbitrary distributions and extend those results to get a product-type formula for the principal symbol of an FIO applied to an arbitrary distribution. The principal term can then be ``peeled off'' to obtain the lower order amplitudes by a similar procedure.

 Most proofs are microlocal to avoid using unique continuation results, but we require an important geometric assumption, which is an \emph{extended convex foliation condition} (see Section \ref{foliation_defn} for the smooth setting) for each wave speed $c_\PS$. As mentioned in \cite{SUVRigidity}, for a particular wave speed, this condition relates to the existence of a function with strictly convex level sets, which in particular holds for simply connected compact manifolds with strictly convex boundaries such that the geodesic flow has no focal points (lengths of non-trivial Jacobi fields vanishing at a point do not have critical points), in particular if the curvature of the manifold is negative (or just non-positive). Thus, caustics are still allowed under this condition since a manifold with non-positive curvature satisfies the condition.
 Also, as explained in \cite{SUV-ElasticLocalRay}, if $\Omega$ is a ball and the speeds increase when the distance to the center decreases (typical
for geophysical applications), the foliation condition is satisfied.

  We denote by $u_h$ the solution to the homogeneous elastic equation on $\RR^3$, with transmission conditions and initial time Cauchy data $h$ (see \eqref{Elastic_eq}and \eqref{e: trans conditions}). All of our function spaces are of the form $X(\cdot; \mathbb{C}^3)$ since we have vector valued functions in the elastic setting, but throughout the paper, we will not write the vector valued part $\mathbb{C}^3$ to make the notation less burdensome. Let $\overline{\Omega}^c$ be the complement of $\overline{\Omega}$ and we define the \emph{exterior measurement operator} $\mathcal F: H^1_c(\overline{\Omega}^c) \oplus L^2_c(\overline{\Omega}^c) \to C^0(\RR_t; H^1(\overline{\Omega}^c)) \cap  C^1(\RR_t; L^2(\overline{\Omega}^c))$ as (see (\ref{e: outside meas operator}) for full details)
\[
\mathcal F: h_0 \to u_{h_0}(t)|_{\overline{\Omega}^c}.
\]
  The operator $\mathcal F$ only measures waves outside $\Omega$ after undergoing scattering within $\Omega$, and it is associated to a particular elastic operator $\Op$ with a set of parameters. Given a second set of elastic parameters $\tilde\lambda$, $\tilde\mu, \tilde \rho$ we obtain analogous operators $\tilde \Op$ and $\tilde{\mathcal F}$. Denote the associated $\PS$ wave speeds $c_{\PS}$ and $\tilde c_{\PS}$. The goal of the work is to prove unique determination of $\mu, \lambda, \rho$ from $\mathcal F$ under some geometric hypotheses.
   From here on, we use $\PS$ to refer to either subscript or wave speed. In addition, to avoid the technical difficulties of dealing with corners or higher codimension singularities of $c_{\PS}$, we always assume that the singular support of $c_{\PS}, \tilde c_{\PS}$ lies in a closed, not necessarily connected hypersurface in $\Omega$; we will deal with corners and edges in a separate paper.

 We assume the Lam\'{e} parameters $\lambda(x)$ and $\mu(x)$ satisfy the \emph{strong convexity condition}, namely that $\mu>0$ and $3\lambda+2\mu >0$ on $\overline{\Omega}$. We also assume that the parameters $\lambda, \mu, \rho$ lie in $L^{\infty}(\RR^3)$ and that $\lambda, \mu, \rho$ are piecewise smooth functions that are singular only on a set of disjoint, closed, connected, smooth hypersurfaces $\Gamma_i$ of $\overline{\Omega}$, called \emph{interfaces}. We let $\Z = \bigcup \Gamma_i$ be the collection of all the interfaces.
 The two wave speeds are $c_P = \sqrt{(\lambda + 2\mu)/\rho}$ and $c_S = \sqrt{\mu/\rho}$, where $\rho$ is the density. In particular, this ensures that $c_P > c_S$ on $\overline{\Omega}$. As in \cite{CHKUControl}, we will probe $\Omega$ with Cauchy data (an \emph{initial pulse}) concentrated close to (but outside $\bar \Omega$) $\Omega$ with a particular polarization. We will denote by
\[
g_{P} = c_{P}^{-2}dx^2, \qquad g_{S} = c_{S}^{-2}dx^2
\]
the two different (rough) metrics associated to the rays. As in \cite{CHKUControl}, we can define the distance functions $d_{\PS}(\cdot, \cdot)$ corresponding to the respective metrics by taking the infimum over all lengths of the piecewise smooth paths between a pair of points. 

Let us define the closed subset $\mathcal{D}:= \{ x \in \Omega ;\ c_P(x) = 2c_S(x)\} \subset \Omega$. In this article we prove the following result.
\begin{theorem}\label{Main_Th_1}
	Assume $\mathcal F = \tilde {\mathcal F} $, and that  $c_{\PS}$, $\widetilde{c}_{\PS}$ satisfy the extended geometric foliation condition (see Section \ref{s: foliation condition}). If $(\lambda,\mu,\rho)=(\tilde{\lambda},\tilde{\mu},\tilde{\rho})$ in $ \Omega^c$,
then $(\lambda,\mu,\rho)=(\tilde{\lambda},\tilde{\mu},\tilde{\rho}) \quad \mbox{inside }\Omega\setminus\mathcal{D}.$

\end{theorem}

\begin{rem}
We use the map $\mathcal F$ for mathematical convenience but we could also consider the time dependent Dirichlet-to-Neumann map on a bounded domain for the elastic wave equation instead. As shown in \cite{Rach00} and \cite{SUV2019transmission}, the Dirichlet-to-Neumann map determines $\mathcal F$ since one can uniquely determine all three material parameters and their derivatives at the boundary from the Dirichlet-to-Neumann map, which then allows one to smoothly extend them to all of $\RR^3$.
\end{rem}

\begin{rem}
	We do not need the data to be measured for the whole time $\RR_t$; instead it is sufficient to take measurement for a finite time $(0,2T)$ where $T>0$ is the maximum time that a $S$-wave takes to travel past $\Omega$. Observe that since $c_P>c_S$ therefore, $T>0$ is enough time for $P$-waves to travel past $\Omega$. We will give the precise definition of $T$ in Section \ref{s: foliation condition}.
\end{rem}



\subsection*{Elastic wave equation and transmission conditions}
In this section, let us give the basic definitions and setup of the elastic wave equation that we study in the main theorem. Recall that $\Omega \subset \R^3$ is a smooth, bounded domain. We consider two extensions of our domain $\Omega\subset \R^3$ as $\Omega \Subset \Theta \Subset \Upsilon \Subset \R^3$, where $A\Subset B$ denotes that $A$ is compactly contained in $B$.
We probe $\Omega$ with an initial pulse supported in $\Theta\setminus\overline{\Omega}$, where $\Theta$ is a small neighbourhood of $\overline{\Omega}$, and a control that is supported in $\Upsilon\setminus\overline{\Omega}$, which is a rather large region in $\R^3$. For an incoming elastic wave $u_I$ that enters $\Omega$, we obtain an array of output scattered waves outside $\Omega$ after hitting one or more interfaces. From the measurements of the input and the scattered outputs, the goal is to recover the piecewise smooth coefficients in $\Omega$.

Define the \emph{$P$-depth} $d^*_{\Theta}(x)$ of a point $x$ inside $\Theta$:
\[
d^*_{\Theta}(x) = \begin{cases}
+d_P(x,\p \Theta), & x \in \Theta ,\\
-d_P(x, \p \Theta), & x \notin \Theta.
\end{cases}.
\]
We use the (rough) metric $g_P$ since finite speed of propagation for elastic waves is based on the faster $P$-wave speed. We will add to the initial pulse a Cauchy data control (a \emph{tail}) supported outside $\Theta$, whose role is to remove multiple reflections up to a certain depth, and is controlled by a time parameter in $(0,\text{diam}_P\Omega)$. Since we will probe all of $\Omega$, fix any $T >\text{diam}_S\Omega$. We only require controls supported in a sufficiently large Lipschitz neighborhood $\Upsilon\subsetneq\RR^3$ of $\overline{\Theta}$ that satisfies $d_S(\partial \Upsilon, \overline{\Theta})>2T$ and is otherwise arbitrary. Thus, we only require elastic wave solutions on the finite time interval $(0, 2T)$.

Define the Neumann operator at $\Gamma$ as the normal component of the stress tensor, given as
\begin{equation}\label{elastic_Neumann}
\mathcal N_{\pm} u = (\lambda \div \otimes \text{I} + 2\mu \hat \nabla)u \cdot \nu|_{\Gamma_\pm},
\end{equation}
where $\nu$ is a fixed unit normal vector at $\Gamma$.
The Cauchy problem for the isotropic, inhomogeneous elastic wave equation we consider is
\begin{equation}\label{Elastic_eq}
\begin{aligned}
\Op u(t,x) =& 0,\qquad &&\mbox{in }\RR_t\times\R^3,\\
u|_{t=0} = \psi_0,\quad \p_t u|_{t=0} =& \psi_1, \qquad &&\mbox{in }\R^3,
\end{aligned}
\end{equation}
where $\psi_0,\psi_1$ are compactly supported in $\mathbb R^3\setminus\overline{\Omega}$. We also impose the transmission conditions at each interface $\Gamma_i$
\begin{equation}\label{e: trans conditions}
[u] = [\mathcal N u] = 0 \qquad \text{ on } \Gamma_i,
\end{equation}
where $[v]$ stands for the jump of $v$ from the exterior to the interior across $\Gamma_i$.
Let us also define the function spaces which will be useful in the analysis.
Define the space for the Cauchy data
\begin{equation}\label{Cauchy_data_set}
\C := H^1_0(\Upsilon;\mathbb{C}^3) \oplus L^2(\Upsilon;\mathbb{C}^3).
\end{equation}
Next, define $F$ to be the solution operator for the elastic wave initial value problem:
\begin{equation}
F: H^1(\RR^n)\oplus L^2(\RR^3) \to C(\RR, H^1(\RR^3))
\qquad
F(h_0,h_1)=u \text{ s.t. }
\begin{cases}
	\Op u &= 0,\\
				u\restriction_{t=0} &= h_0,\\
			\p_t u\restriction_{t=0} &= h_1.
	\end{cases}
\end{equation}
Thus, the outside measurement map may be written as
\begin{equation}\label{e: outside meas operator}
\mathcal F(h) = F(h)(t)|_{\overline \Omega^c}, \qquad h \in \mathbf C.
\end{equation}
Let $R_s$ propagate Cauchy data at time $t=0$ to Cauchy data at $t=s$:
\begin{equation}
R_s = (F, \p_t F)\Big|_{t=s}:H^1(\RR^3) \oplus L^2(\RR^3) \to H^1(\RR^3)\oplus L^2(\RR^3).
\end{equation}

\section{Geometric assumptions and notation}
 In this section, we introduce our notation and state the geometric assumptions we make.
All the definitions and results in this section are a brief summary of what was defined in \cite{B_density,CHKUControl,CHKUElastic}. Recall that $\Z$ is the collection of closed connected hypersurfaces $\Gamma_j$, for $j=0,1,\dots,n$ in $\overline{\Omega}$ where $\Gamma_0 = \partial\Omega$ and the parameters are smooth on each connected component of $\R^3\setminus\Gamma$.

The principal symbol of the hyperbolic operator $\Op$ is given as
\begin{equation}\label{princ_symb}
p(t,x,\tau,\xi) = -\rho\left[\left(\tau^2 - c_S^2|\xi|^2\right)I - \left(c_P^2-c_S^2\right)(\xi\otimes\xi)\right].
\end{equation}
One can calculate the lower order terms in the full symbol of $\Op$, and the order one term is 
\begin{equation}\label{sub-princ_symb}
p_1(t,x,\tau,\xi) = -i\left[\nabla_x \lambda \otimes\xi +(\nabla_x\mu\cdot\xi)I + \xi\otimes\nabla_x\mu\right].
\end{equation}
From the principal symbol, we readily observe that $\xi$ and $\xi^{\perp}$
are eigenvectors of $p(t,x,\tau,\xi)$ with eigenvalues $c_P$ and $c_S$ respectively. The eigenspace corresponding to $c_P$ is one dimensional whereas the eigenspace corresponding to the eigenvalue $c_S$ is two dimensional.

\subsection*{Geodesics and bicharacteristics}
The bicharacteristics curves $\gamma^{\pm}_{\PS}$ in $T^*(\R\times\Ups)$ are the integral curves of the Hamiltonian vector fields $V_{H^{\pm}_{\PS}}$, where $H^{\pm}_{\PS} = \tau \pm c_{\PS}|\xi|$ along with the condition that $\gamma^{\pm}_{\PS}$ lies in the set $\{\text{det }p(t,x,\tau,\xi) = 0\}$.
Parametrized by $s\in \R$, one obtains
\begin{equation}\label{parametrization_bichar}
\frac{dt}{ds} = c_{\PS}^{-1}, \qquad
\frac{dx}{ds} = \pm \frac{\xi}{|\xi|}, \qquad
\frac{d\tau}{ds} = 0, \qquad
\frac{d \xi}{ds} = \mp |\xi| \nabla_x (\log c_{\PS}),
\end{equation}
along with the condition that $\text{det }p(t,x,\tau,\xi) = 0$.
We refer to $\gamma_{\PS}^{\pm}$ as the forward and backward $\PS$ bicharacteristic curves.

\begin{definition}
	We refer to a piecewise smooth curve $\gamma_{\PS}:I\subset \R \to \overline{\Omega}$ as a unit speed broken geodesic in $(\overline{\Omega},g_{\PS})$ if
	\begin{enumerate}
		\item $\gamma_{\PS}$ is a continuous path which is smooth in $\overline{\Omega}\setminus\Gamma$,
		\item each smooth segment of $\gamma_{\PS}$ is a unit speed geodesic with respect to $g_{\PS}$,
		\item $\gamma_{\PS}$ intersects $\Gamma$ at only finitely many points $t_i \in I$ and all the intersections are transversal,
		\item $\gamma_{\PS}$ obeys Snell's law of refraction for elastic waves where it cuts $\Gamma$.
	\end{enumerate}
A \emph{broken bicharacteristic} is a path in $T^*\RR^n$ of the form $(\gamma,\gamma'^\flat)$, the flat operation taken with respect to $g_P$ or $g_S$ as appropriate. Note that a broken geodesic defined this way may contain both $P$ and $S$ geodesic segments. More precisely, a broken bicharacteristic (parameterized by a time variable) can be written as $\gamma: (t_0,t_1) \cup (t_1, t_2) \cup \dots \cup (t_{k-1},t_k) \to T^*\RR^n \setminus \Gamma$, which is a sequence of bicharacteristics connected by reflections and refractions obeying Snell's law: for $i = 1, \dots, k-1$,
\begin{equation}
\gamma(t_i^-), \gamma(t_i^+) \in T_{\Gamma}^*(\mathbb R^n),
\qquad \qquad (d\iota_\Gamma)^*\gamma(t_i^-) = (d\iota_\Gamma)^*\gamma(t_i^+),
\end{equation}
where $\iota_\Gamma: \Gamma \hookrightarrow \RR^n$ is the inclusion map and $\gamma(t_i^\mp) = \lim_{t \to t_i^\mp}\gamma(t)$. We always assume that $\gamma$ intersects the interfaces transversely since for our parametrix construction, we assume that solutions have wave front set disjoint from bicharacteristics tangential to some of interfaces. Each restriction $\gamma|_{(t_i,t_{i+1})}$ is a $P$-\emph{bicharacteristic}, respectively $S$-\emph{bicharacteristic} if it is a bicharacteristic for $\p^2_t - c_P \Delta$, respectively $\p^2_t - c_S \Delta$. We also refer to each such bicharacteristic as a \emph{branch} of $\gamma$; we are sometimes more specific and write $P$ branch or $S$ branch if we want to specify the associated metric. For each $i$, note that $\gamma(t_i)$ projected to the base manifold is a point of $\Gamma_{k_i}$ for some $k_i$. A branch $\gamma|_{(t_i,t_{i+1})}$ is \emph{reflected} if the inner product of $\gamma'(t_i^+)$ and $\gamma'(t^-_i)$ (when projected to base space) with a normal vector to $\Gamma_{k_i}$ have opposite signs. Otherwise, it is a \emph{transmitted branch}. Say that $\gamma|_{(t_i,t_i+1)}$ is a \emph{mode converted branch} if it is a $\PS$ branch and $\gamma|_{(t_{i-1},t_i)}$ is a $S/P$ branch.

A \emph{purely transmitted P/S broken geodesic} (a concatenation of smooth $P$ or $S$ geodesics) is a unit-speed broken geodesic that consists of only $\PS$ transmitted branches; that is, the inner products of $\gamma'(t_i^-)$ and $\gamma'(t_i^+)$ with the normal to $\Gamma$ have identical signs at each $t_i$ and they are all either $P$ geodesics or $S$ geodesics. A \emph{purely transmitted P/S broken bicharacteristic} is then defined the same way using projection to base space.

\end{definition}
Note that unlike \cite{CHKUElastic}, these unit speed broken geodesics are either purely $P$-geodesics or purely $S$-geodesic. We define the broken $\PS$ bicharacteristics as $(\gamma_{\PS},\dot\gamma_{\PS}^{\flat})$ where $\gamma$ is a unit speed $\PS$ broken geodesic and the operation `$\flat$'(flat) is taken with respect to the metric $g_{\PS}$.

\subsection*{Foliation Condition}\label{s: foliation condition}

We assume that the domain $\Omega$ has an \emph{extended convex foliation} with respect to both of the metrics $g_{P}$ and $g_S$. This is an extension of the convex foliation condition given in \cite{UVLocalRay} to the piecewise smooth setting and was introduced in \cite[Definition 3.2]{CHKUElastic}.
\begin{definition}[Extended convex foliation]\label{foliation_defn}
	We say  $\sig : \overline{\Omega} \mapsto [0,\tilde{q}]$ is a (piecewise) extended convex foliation for $(\Omega,g_{\PS})$ if
	\begin{enumerate}
		\item $\sig$ is smooth and $d\sig \neq 0$ on $\overline{\Omega}\setminus\Gamma$.
		\item $\sig$ is upper semi-continuous.
		\item each level set $\sig^{-1}(q)$ is geodesically convex with respect to $g_{P}$ and $g_S$, when viewed from $\sig^{-1}((q,\tilde q))$ for any $t \in [0,\tilde{q})$.
		\item $\partial\Omega = \sig^{-1}(0)$ and $\sig^{-1}(\tilde{q})$ has measure zero.
		\item there is some $q_i \in [0,\tilde{q}]$ such that $\Gamma_i \subset \sig^{-1}(q_i)$ for $i=0,\dots,m$.
		\item $\limsup_{\epsilon \to 0+}c_{\PS}|_{\sig^{-1}(q+\epsilon)}
		\leq \limsup_{\epsilon \to 0^+} c_{\PS}|_{\sig^{-1}(q-\epsilon)}$ whenever $\Gamma_i \subset \sig^{-1}(q)$ for some $i$ and $\Gamma_i$.
	\end{enumerate}
We say $(c_P,c_S)$ satisfies the \emph{extended foliation condition} if there exists an extended convex foliation for $(\Omega, c_\PS)$.
\end{definition}

See the discussion below \cite[Definition 3.2]{CHKUElastic} for an explanation of the last condition. We write $\Omega_{q}:= \sig^{-1}(q,\tilde{q}]$ to denote the part of the domain whose boundary is $\Sigma_q:= \sig^{-1}(q)$.
Let us fix the convention of writing `above' $\Sigma_q$ to be outside of $\Omega_q$ and `below' to be inside $\Omega_q$.
We write $\Sigma^{\pm}_q$ to denote two copies of $\Sigma_q$ approached from 'above' or `below' $\Sigma_q$.
Let us observe that, if required, we can extend each $\Gamma_j$ along with $\Sigma_{q_j}$ and denote $\Omega_j$ to be the connected components of $\overline{\Omega}\setminus \Gamma$. Write $\widetilde{\Omega}_j$ to be $\Omega_{q_j}$, where $\Gamma_j\subset\Sigma_{q_j}$.

\begin{definition}
	We define the set of the inward or the outward pointing covectors at a closed connected hypersurface $\Sigma$ as
	\begin{equation*}
	T_{\pm}^*\Sigma := \{ (x,\xi) \in T^*\Sigma : \pm \langle \xi,d\sig\rangle > 0 \},
	\end{equation*}
	where the above inner product is taken in the Euclidean sense.
	
	For a domain $\Omega_{q}$ we define
	\begin{equation*}
	\partial_{\pm}T^*\Omega_{q} := \{ (x,\xi) \in T^*\Omega ; \pm \langle \xi,d\sig\rangle > 0 \}.
	\end{equation*}
\end{definition}

With the help of the notion of the inward or outward covectors, we define the \emph{\q-interior travel time} and \emph{lens relation} in $\Omega$.
\begin{definition}
	Let $(x,\xi) \in T_{+}^*\Sigma_{\q}\setminus\{0\}$. Let $\gamma_{\PS}$ be the unit speed broken geodesic such that
	\begin{equation*}
	\lim\limits_{t\to 0^+}\left(\gamma_{\PS}(t),\dot{\gamma}_{\PS}(t)\right) = (x,\xi).
	\end{equation*}
	\begin{enumerate}[(i)]
		\item  We define the $q$-interior travel time $l_{\PS,q}(x,\xi)>0$ such that $\gamma_{\PS}(l_{\PS,\q}) \in \Sigma_{\q}$.
		\item The $q$-interior $\PS$ lens relation is $L_{\PS,\q}(x,\xi) := \left(\gamma_{\PS}(l_{\PS,\q}),\dot{\gamma}_{\PS}(l_{\PS,\q})\right) \in T^*_{-}\Sigma_{\q}$, where $l_{\PS,\q}(x,\xi)$ is defined as above.
	\end{enumerate}
\end{definition}

We end this section by summarizing all the assumptions we have made so far.
\begin{assumption}
	Let us assume all the notations and definitions above. We assume that:
	\begin{enumerate}[(i)]
		\item the interfaces $\Gamma_j$ are a collection of disjoint, connected, closed hypersurfaces in $\overline{\Omega}$,
		\item $\Omega$ has an extended convex foliation $\sig$ with respect to both the metrics $g_{\PS}$.
		
	\end{enumerate}
\end{assumption}

\subsection*{FIOs and elastic parametrix}
In this section we show the microlocal parametrix construction for the system \eqref{elastic_Neumann} in the smooth setting similar to \cite{RachBoundary}. Since many of our argument in the proof of the main theorem are local, this will suffice. The full parametrix with transmission conditions where the material parameters have discontinuities is in Appendix \ref{a: elastic parametrix}.
First, the geometric optics solution for the initial value problem \eqref{Elastic_eq} when the material parameters are smooth has the form
\begin{equation*}
U = E_0 f_0 + E_1 f_1,
\end{equation*}
where $E_k$, $k=0,1$ are the solution operators given in terms of Fourier Integral Operators (FIOs). We write $A\equiv B$ for two FIOs $A$ and $B$ to denote that $A$ is same as $B$ modulo a smoothing operator.
We impose that the FIOs $E_0$ and $E_1$ solve the following system modulo a smoothing operator.
\begin{equation}\label{E_k PDE}
\begin{aligned}
\Op E_k &\equiv 0 \quad &&\mbox{on } (0,\infty) \times \R^3,\\
E_0|_{t=0} &\equiv I, \quad \partial_t E_0|_{t=0} \equiv 0, \quad&& \mbox{on }\R^3,\\
E_1|_{t=0} &\equiv 0, \quad \partial_t E_1|_{t=0} \equiv I, \quad&& \mbox{on }\R^3.
\end{aligned}
\end{equation}
The FIOs $E_k$ for $k=0,1$ are given as
\begin{equation}\label{E_k FIO}
\begin{aligned}
E_k v
=&\sum_{\pm, p/s, l} \int_{\R^3} e^{i\phi^{\pm}_{\PS}(t,x,\xi)} a^{\cdot,l}_{\pm,k,\PS}(t,x,\xi) \hat{v}_l(\xi)\,d\xi, \quad \mbox{where }v = (v_1,v_1,v_3).
\end{aligned}
\end{equation}
%
The phase function $\phi^{\pm}_{\PS}(t,x,\xi)$ is homogeneous of order 1 in $\xi$ and solves the eikonal equation
\begin{equation}\label{Eqn for phi_1}
\det{p(t,x,\partial_t \phi^{\pm}_{\PS},\nabla_{x}\phi^{\pm}_{\PS})} = 0.
\end{equation}
One can simplify the eikonal equation to
\begin{equation}\label{Eqn for phi_2}
\partial_t \phi^{\pm}_{\PS} = \mp c_{\PS} \lvert \nabla_{x}\phi^{\pm}_{\PS} \rvert.
\end{equation}
We can choose the initial value to be $\phi^{\pm}_{\PS}|_{t=0} = x \cdot \xi$ and solve the above equation using Hamilton-Jacobi theory.

\begin{rem}
	Note that the phase function $\phi^{\pm}_{\PS}$ can be determined by the principal symbol $p$ and thus, the wave speeds $c_{\PS}$. The principal symbol of $E_k$ is also determined by the wave speeds and not the density \cite{RachBoundary}. Therefore, if one needs to recover the parameters $\rho,\lambda,\mu$ individually, then one must consider the lower order terms of the asymptotic expansion of $a^{\cdot,l}_{\pm,k,\PS}$.
\end{rem}
Unlike the phase function $\phi^{\pm}_{\PS}$, the amplitudes cannot be determined using only the principal symbol $p$.
We do an asymptotic expansion of the amplitudes $a^{\cdot,l}_{\pm,k,\PS}(t,x,\xi)$ as
\begin{equation*}
a^{\cdot,l}_{\pm,k,\PS}(t,x,\xi)
= \sum_{J = 0, -1, -2, \dots}\left(a^{j,l}_{\pm,k,\PS}\right)_{J},
\end{equation*}
where $(a^{\cdot,l}_{\pm,k,\PS})_{J}$ are homogeneous of order $J$ in $\lvert \xi \rvert$.
Now, each $3\times 3$ matrices $\left(a^{\cdot,l}_{\pm,k,\PS}\right)_J(t,x,\xi)$ satisfies
\begin{equation}\label{Eqn for a-J}
\begin{aligned}
p(t,x,\partial_t \phi^{\pm}_{\PS},\nabla_{x}\phi^{\pm}_{\PS})(a_{\pm,k,\PS})_{J-1}
= B_{\PS}(a_{\pm,k,\PS})_{J} + C_{\PS}(a_{\pm,k,\PS})_{J+1},
\end{aligned}
\end{equation}
where $(a_{\pm,k,J})_1 = 0$ and
the matrix operators $B_{\PS}$, $C_{\PS}$ are given as
\begin{equation}\label{Terms B C}
\begin{aligned}
(B_{\PS}V)
&= i\partial_{\tau,\xi}p\left(t,x,\partial_t \phi^{\pm}_{\PS}, \nabla_x \phi^{\pm}_{\PS}\right) \cdot \partial_{t,x}V
+ i p_1(t,x,\partial_t \phi^{\pm}_{\PS}, \nabla_x \phi^{\pm}_{\PS})V\\
& \qquad \qquad + i\frac{1}{2}\sum_{\lvert \beta \rvert = 2} \partial^{\beta}_{\tau,\xi} p\left(t,x,\partial_t \phi^{\pm}_{\PS}, \nabla_x \phi^{\pm}_{\PS}\right) \cdot \left( \partial^{\beta}_{t,x} \phi^{\pm}_{\PS} \right)V,\\
\left(C_{\PS}V\right)
&= i\partial_{\tau,\xi}p_1\left(t,x,\partial_t \phi^{\pm}_{\PS}, \nabla_x \phi^{\pm}_{\PS}\right) \cdot \partial_{t,x}V\\
& \qquad \qquad + \frac{1}{2}\sum_{\lvert \beta \rvert = 2}\partial^{\beta}_{\tau,\xi}p\left(t,x,\partial_t \phi^{\pm}_{\PS}, \nabla_x \phi^{\pm}_{\PS}\right) \cdot \partial^{\beta}_{t,x} V.
\end{aligned}
\end{equation}
In order to calculate the explicit form of $(a^{j,l}_{\pm,k,\PS})_{J},$ let us define the unit vectors $N= \frac{\nabla_x \phi^{\pm}_{P}}{\lvert \nabla_x \phi^{\pm}_{P}\rvert}$ and define $N_1$, $N_2$,
such that $\{N_1, N_2\}$ forms an orthonormal basis of the kernel of $p(t,x,\partial_t \phi^{\pm}_{S},\nabla_{x}\phi^{\pm}_{S})$.
Observe that the unit vector $N$ spans the kernel of $p(t,x,\partial_t \phi^{\pm}_{P},\nabla_{x}\phi^{\pm}_{P})$ and $\{N_1, N_2\}$ form an orthonormal basis for the kernel of $p(t,x,\partial_t \phi^{\pm}_{S},\nabla_{x}\phi^{\pm}_{S})$.
Now, let us write
\begin{equation}\label{asym_exp}
\begin{aligned}
(a^{\cdot,l}_{\pm,k,P})_{J}
&= (h^{\cdot,l}_{\pm,k,P})_{J} + (b^{l}_{\pm,k,P})_{J}N, \qquad l =1,2,3,\\
(a^{\cdot,l}_{\pm,k,S})_{J}(t,x,\xi)
&= (h^{\cdot,l}_{\pm,k,S})_{J}
+ \left[(b^{l}_{\pm,k,S})_{1,J}N_1 + (b^{l}_{\pm,k,S})_{2,J}N_2\right],
\end{aligned}
\end{equation}
where $(h^{\cdot,l}_{\pm,k,\PS})_{J}$ is perpendicular to the kernel of $p(t,x,\partial_t \phi^{\pm}_{\PS}, \nabla_x \phi^{\pm}_{\PS})$ for $J \leq -1$ with $(h^{\cdot,l}_{\pm,k,\PS})_{0} = 0$.

Observe that from \eqref{Eqn for a-J} we obtain a necessary condition that
\begin{equation}\label{Compatibility condition}
N_{\PS} \left[ B_{\PS}(a_{\pm,k,\PS})_{J} + C_{\PS}(a_{\pm,k,\PS})_{J+1} \right] = 0, \quad \mbox{for } J=0,-1,-2,\dots,
\end{equation}
where $N_P = N$ and $N_S = N_1, N_2$.

\section{Symbol of an FIO applied to an arbitrary distribution} \label{s: Weinstein calculus}
In a layer stripping procedure to recover the density of mass, we will use scattering control to generate internal sources with a specified wavefront set (as done in \cite{SUV2019transmission,CHKUElastic}). The data that we will be able to recover from the outside measurement operator has the form $FV$ where $V$ is more or less an arbitrary distribution (the internal source we want to generate) and $F$ is a Fourier integral operator representation of the elastic wave propagator. We need to recover the terms in the polyhomogeneous expansion of the symbol of $F$ from such data, and in particular show that they vanish when doing a uniqueness argument. The main idea is quite natural and an analogous argument can be found in \cite[Example 2.6]{Dencker-Polarization}. When $A$ is a pseudodifferential operator and $u$ is a compactly supported distribution that is smooth outside $\{0\}$, $\widehat{Au}(\xi)$ has an expansion involving the full symbol of $A$ and and its derivatives, together with $\hat u$ and its derivatives, evaluated at $(0,\xi)$. When $Au \in C^\infty$, then the terms in this expansion vanish, which creates algebraic equations where one can extract information regarding the principal symbol of $A$ as well as its lower order symbols in the classical symbol expansion. Our aim is to do something similar when $A$ is a Fourier integral operator whose associated Lagrangian is a canonical graph, and the formalism of Weinstein's symbol calculus in \cite{weinstein1976} will be natural to obtain a similar expansion by associating a principal symbol to an arbitrary distribution.

In a seminal paper \cite{weinstein1976}, Weinstein showed how to define the principal symbol of an arbitrary distribution and showed a product type formula for the principal symbol of a PsiDO applied to an arbitrary distribution. Here, we tweak several argument in \cite{weinstein1976} to compute the principal symbol of a canonical graph FIO applied to a distribution. We state the necessary lemmas and propositions we will use in this manuscript but postpone the proofs to Appendix \ref{s: proofs of Weinstein symbol stuff} to not interrupt the flow of the article.

\subsection{Weinstein symbols of arbitrary distributions}
We use the formulation of Weinstein to define the symbol of a distribution. First, fix $(x_0, \xi) \in T^*\RR^n$. Let $\g \in \mathcal D'(\RR^n)$ and let $\phi: \RR^n \to \RR$ be a $C^\infty$ function with $\phi(x_0) = 0$ and $d\phi(x_0) = \xi$. We define the distribution $\g_\phi^\tau$ for $\tau \geq 1$ by
\[
\langle \g_\phi^\tau, u \rangle
= \langle \g, \tau^{n/2} e^{-i \tau \phi(x)} u( \sqrt \tau (x-x_0)) \rangle
\]
where $u$ is a test function.
If $\g \in L^1$, then $\g^\tau _\phi =( \g e^{-i \tau \phi})((x-x_0)/\sqrt{\tau})$.
Essentially, $\g^\tau_\phi = T_{x_0}D_{1/\sqrt \tau} M_{e^{-i\tau \phi}}  \g$ where $M_\bullet$ is multiplication operator, $D_\bullet$ is dilation operator, and $T_{x_0}$ is a translation operator so this distribution is well defined. We will usually use $\phi = (x - x_0) \cdot \xi$ so we leave out the subscript $\phi$ in this case and just write $g^\tau$. It is also useful to denote $u_\tau  = \tau^{n/2} e^{-i \tau \phi(x)} u( \sqrt \tau (x-x_0))$ so that $\langle \g^\tau , u \rangle = \langle \g , u^\tau \rangle.$

We also need a notion of ``order'' of a distribution by measuring the growth rate of $\g^\tau$. We have the following definition from \cite[Definition 1.1.3, 1.1.7]{weinstein1976}.

\begin{definition}\label{d: order of distribution}
If $\mathfrak E$ is an vector space with a distinguished class of subsets called ``bounded sets'', we denote by $S^N(\mathfrak E)$ the set of families $[1, \infty) \ni \tau \mapsto \g^\tau \in \mathfrak E$ for which the set $\{ \tau^{-N} \g^\tau | \tau \geq 1\}$ is bounded. If $\tau \mapsto \g^\tau$ belongs to $S^N(\mathfrak E)$, we write $\g^\tau = O(\tau^N)$.
If the set $S = \{ N | \g^\tau = O(\tau^N)\}$ is of the form $[a, \infty)$, we define \emph{the order of} $\g$ at $(x_0, \xi)$ to be $a$ and denote it by $O_{x_0, \xi_0}(\g)$. If $S = (-\infty, \infty)$, we define $O_{x_0} (\g) = -\infty$.

 Let $ \g \in \mathcal D'(\RR^n)$ be such that $O_{(0,\xi)}(\g) \leq N$. We define the \emph{Weinstein symbol} of order $N$
\[
\sigma^N_{(x_0,\xi)}(\g)
\]
 of $\g$ at $(x_0,\xi) \in T^*\RR^n$ is
 the image in $S^N(\mathcal D'(\RR^n)) / S^{N-1}(\mathcal D'(\RR^n))$
of the family $\g_\phi^\tau$ under the natural map
$S^N(\mathcal D'(\RR^n)) \to S^N(\mathcal D'(\RR^n)) / S^{N-1}(\mathcal D'(\RR^n))$.
\end{definition}
This is a slight deviation from \cite[Definition 1.3.6]{weinstein1976} which instead involves $\phi$ and certain equivalence classes, but as mentioned there, in computing $\sigma^N_{(x_0, \xi)} (\g)$ for a specific $\g$, it suffices to determine the behavior modulo $S^{N-1}(\mathcal D'(\RR^n))$ of the family $\g^\tau_\phi$ for a particular $\phi$ with $d\phi(x_0) = \xi$. In this paper, we will only use $\phi(x) = (x-x_0) \cdot \xi$.

Also, following \cite[Definition 1.3.12]{weinstein1976}, a given distribution is \emph{homogeneous of degree $N$} at $(x_0, \xi)$ if for a particular $\phi$ with $d\phi(x_0) = \xi$, $\g_\phi^\tau = \tau^N \gamma + O(\tau^{N-1/2})$ where $\gamma$ is a fixed element of $\mathcal D' (\RR^n)$. This is equivalent to the definition in \cite{weinstein1976} and then the symbol of $g$ can be identified with $\gamma$ as described there. For the distributions we consider here, it will suffices to pair $\g^\tau$ with an arbitrary test function and determine the leading order term in $\tau$ as $\tau \to \infty$.

As described in \cite[Section 1.6]{weinstein1976}, the above definitions and concepts naturally extend to vector-valued distributions on a manifold using natural identifications \[
\mathcal D'(\RR^n, V) \approx \mathcal D'(\RR^n) \otimes_{\mathbb C} V\]
where $V$ is a finite dimensional vector space over $\mathbb C$.

\subsection{Symbol of a PsiDO applied to a distribution}
Here we state few results which are minor modifications of Results in \cite{weinstein1976},
whose proofs we provide in Appendix \ref{s: proofs of Weinstein symbol stuff}.

 \begin{lemma}  \label{l:PSIDO applied to u_tau}
 Let $P \in \Psi^m$ be properly supported with a principal symbol representative $p_m(x, \xi)$. Let $\g \in \mathcal D'(\RR^n)$ and $(x_0, \xi) \in T^* \RR^n \setminus 0$ with $O_{(x_0, \xi)} (\g) \leq N$.
Then $O(P\g) \leq m+N$ and
\[
\sigma^{N+m}_{(x_0,\xi)}(P\g) =
p_m(x_0, \tau \xi) \sigma^N_{(x_0,\xi)}(\g).
\]
In fact, given a test function $u \in \mathcal D(\RR^n)$, we have
\[
Pu_\tau (x) = p_m(x_0, -\tau \xi) u_\tau(x) + O(\tau^{m-1/2})
\]
and
\[
\langle (P\g)^\tau , u \rangle=
p_m(x_0, \tau \xi) \langle \g^\tau, u \rangle + O(\tau^{m+N-1/2})
\]
as $\tau \to \infty$.
\end{lemma}

We also need to know how a diffeomorphism transforms the symbol. First, using the proof of \cite[Proposition 1.4.1]{weinstein1976}, we have
\begin{prop}\label{p: pullback of distribution symbol}
Let $\theta: \RR^n \to \RR^n$ be a diffeomorphism such that $\theta(x_0) = y_0$.
Then $O_{y_0,\phi}(\g) \leq N$ implies
\[
(\g \circ \theta)^\tau_{\phi \circ \theta} - f^\tau_\phi \circ d_{x_0}\theta = O(\tau^{N-1/2}).
\]
where $d_{x_0}\theta : \RR^n \to \RR^n$ is the derivative of $\theta$ at $x_0$.
\end{prop}
It will be useful to write out each term in the above proposition. Explicitly, we have
\beq \label{e: g of theta formula}
\tau^{-n/2}\langle (\g \circ \theta)^\tau_{\phi \circ \theta}, u \rangle
= \langle \g \circ \theta, e^{i \tau\phi(\theta(x))}u(\sqrt \tau (x-x_0)) \rangle
=  \langle \g , J_{\theta^{-1}}(y)e^{i \tau\phi(y)}u(\sqrt \tau (\theta^{-1}(y)-x_0)) \rangle,
\eeq
where $J_{h}$ denotes the Jacobian of $h$.
Then modulo $O(\tau^{N-1/2})$, this is equal to
\begin{multline}\label{e: g of dtheta formula}
\tau^{-n/2}\langle \g ^\tau_\phi \circ d_{x_0}\theta, u \rangle
= \tau^{-n/2}\langle \g ^\tau_\phi , J_{\theta^{-1}}(y_0)u((d_{x_0}\theta)^{-1}y) \rangle \\
= \langle \g  , e^{-i\tau \phi(y)}J_{\theta^{-1}}(y_0)u(\sqrt \tau(d_{x_0}\theta)^{-1}(y-y_0)) \rangle.
\end{multline}

For the principal symbol, observe that
\[
\phi \circ \theta(x_0) = 0 \qquad d_{x_0}(\phi(\theta)) = (\theta'_x(x_0))^t\phi'(y_0) = (\theta'_x(x_0))^t\eta_0 = \theta^* \eta_0.
\]
Thus, $\phi\circ \theta$ is a phase function for $(x_0, \xi_0)$ where $\xi_0 = \theta^*\eta_0$. Notice, $\theta$ induces a map on $\mathcal D'(\RR^n)$ by $h \mapsto h \circ d_{x_0}\theta$ which preserves bounded sets. It also induces a map on the phase functions via $[\phi] \mapsto [\phi \circ \theta]$ that induces a map on principal symbols.
\begin{cor}
With the setup as above, we have
\[
\sigma_{x_0, \theta^*\eta_0} (\g \circ \theta)
= \sigma_{y_0,\eta_0 }(\theta)[\sigma_{y_0,\eta_0}(\g)].
\]
\end{cor}

\subsection{Principal symbol of an FIO applied to a distribution}
Let us consider a Fourier integral operator $A \in I^\mu(Y\times X, \Lambda_\chi)$ where $X$ and $Y$ are two manifolds and the associated Lagrangian $\Lambda_\chi$ is a canonical graph of a symplectomorphism $\chi: T^*X \to T^*Y$. In fact, the only  case we need is for $X = \mathbb R^n$ and $Y = \mathbb R^n$, so let us assume this to simplify the notation. Our proofs apply to operators acting on vector bundles as well, and with H\"{o}rmander's notation \cite[Chapter 25]{Hormander_1}, this includes operators in $I^\mu(Y\times X, \Lambda_\chi;\text{Hom}(E,F))$ for vector bundles $E$ and $F$. We can write
\[
\Lambda = \Lambda_\chi = \{ (\chi(x,\xi)), (x, \xi) \} \subset T^*Y \times T^*X.
\]
The order of $A$ is $\mu$ and can be written in the form
\[
A = \int e^{i(S(y,\xi) - x \cdot \xi)} a(y,\xi) \ d\xi
\]
and $a \in S^\mu$.
There is the associated set
\[
C_\Phi = \{ (y,x,\xi): d_\xi \Phi = 0 \}
\]
and a diffeomorphism $T_\Phi: C_\Phi \to \Lambda$. Via a projection, we can view $a$ as a symbol over $C_\phi$ and hence $\Lambda$ via the diffeomorphism $T_\Phi$.
We can write explicitly
\[
\Lambda = \{ y, S'_y, S'_\xi,\xi \}
\]
as the graph of the canonical transformation
\[
\chi: (S'_\xi,\xi) \mapsto  (y, S'_y)
\]
Now, for fixed $\xi_0$, the map \[y \mapsto S'_\xi(y,\xi_0) :=T(y)= x(y)\] is a local diffeomorphism by the assumption we have a canonical graph. Denote the Jacobian determinant $J_T(y) := |\p T/\p y|.$
In the proof below, we will be using the adjoint $A^t$, but the adjoint is also associated to a canonical graph and may be put in the form above, which we shall use.

We aim to prove the following proposition which lets us compute a simple form for the Weinstein principal symbol of an FIO associated to a canonical graph applied to a distribution.

\begin{prop}\label{prop: symbol of FIO applied to distribution}
Let $(x_0, \xi_0) = \chi(y_0,\eta_0)$ and $\g \in \mathcal D'(Y)$. Let $A \in \mathcal I^\mu(X \times Y; \Lambda_\chi)$ so that the distributional adjoint $A^t \in I^\mu(Y \times X; \Lambda_{\chi^{-1}}) $ has a representation of the form
$A^t = \int e^{i(S(y,\xi) - x \cdot \xi)} a^t(y,\xi) \ d\xi$ where $a^t$ has a polyhomogeneous expansion with principal term denoted $a^t_\mu$, homogeneous of degree $\mu$ in the $\xi$ variable. Then
$y_0 = T^{-1}(x_0)$ and
\[
\sigma_{x_0,\xi_0}(A\g)
= a^t_\mu(y_0,-\tau \xi_0)J_{T^{-1}}(x_0) \sigma_{x_0,\xi_0}(\g\circ T^{-1}).
\]

\end{prop}

In this paper, we are interested in $A \in \mathcal I^\mu(X \times Y, \Lambda_\chi)$ with a representation
\[
A = \int e^{i(\phi(x,\eta) - y \cdot \eta)} a(x,\eta) \ d\eta
\]
where $a = a_{(\mu)} + \tilde a$ with $\tilde a \in S^{\mu-1}$ and $a_{(\mu)}$ homogeneous in $\eta$ of degree $\mu$.
Following H\"{o}rmander in \cite[Chapter 25]{Hormander_1}, denoting $s: Y \times X \to X \times Y$ as a function that interchanges the two factors, then $s^*a^t_{prin}$ is the principal symbol of $A^*$ if $a_{prin}$ is the principal symbol of $A$ and $a^t_{prin}$ is the matrix transpose when the vector bundles have been trivialized (in our setting; both vector bundles are $\mathbb{C}^3$ so there is no need to specify a trivialization). As shown in  \cite[Chapter 25]{Hormander_1}, the principal symbol of $A$ is well-defined and determined by a 1/2-density over $\Lambda$ and a Maslov bundle factor that are both determined by $\Lambda$. Hence, $a^t_\mu$ computed in the above proposition and $s^*a^t_{prin}$ restricted to $\Lambda$ can only differ by a factor, denoted $J_\Lambda$, which is completely determined by the Lagrangian $\Lambda$ (since $\Lambda$ is a canonical graph, there is a natural trivialization of $a_{prin}$ described in \cite[Chapter 25]{Hormander_1}). If we denote the polyhomogeneous expansion as
$a \sim \sum_j a_{(\mu-j)}$, then the preceding discussion implies
\[
a^t_\mu|_\Lambda = J_\Lambda s^*a_{(\mu)}^t|_{\Lambda}.
\]
We then get the following important corollary using the definition of the Weinstein symbol.
\begin{cor}\label{cor: symbol of FIO extraction with limits}
With the notation and assumptions in Proposition \ref{prop: symbol of FIO applied to distribution}, and $\g$ a distribution of order $N$, we have
\begin{equation}\label{e: symbol of FIO extraction}
\liminf_{\tau \to \infty}\tau^{-\mu - N}\langle  A \g, u_\tau \rangle
= a^t_{(\mu)}(y_0, \xi_0)J_{\Lambda}(y_0) \tau^{-N}\liminf_{\tau \to \infty}\langle  \g\circ T^{-1}, u_\tau \rangle < \infty
\end{equation}
where $J_{\Lambda}$ is a quantity determined by $\Lambda$.
\end{cor}

\section{Recovery of the parameters } \label{s: recovery of the parameters}
In this section, we will list the preliminary results and ingredients needed to prove Theorem \ref{Main_Th_1} and then give the final proof in Section \ref{s: proof of the main theorem}. Since this is an intricate procedure with various pieces, we provide a summary of the proof.
Note that, with a suitable choice of foliation, we can identify $ \overline{\Omega}$ with the disjoint union
\begin{equation*}
\overline{\Omega} = \bigsqcup_{j=0}^{m} \overline{\Omega}_j,\quad \mbox{where }\Gamma_j \subseteq \Sigma_{q_j} = \sig^{-1}\{q_j\},
\end{equation*}
and $\Omega_j = \sig^{-1}\left([q_{j},q_{j+1}]\right)$ are defined via Definition \ref{foliation_defn}. Here we have $0=q_0<q_1<\dots<q_m=q_{m+1}$ and $|\Omega_m| = 0$.

\subsubsection*{Sketch of the proof of Theorem \ref{Main_Th_1}}
We prove the result using a layer striping argument. We split the proof into two main parts.\\

In the \textbf{first part}, we fix a $q \geq 0$ and assume that the elastic parameters are known on $\Omega_{q}^*:= \Upsilon \setminus \overline{\Omega}_{q}$.
 Using the microlocal parametetrix in the previous section, we construct Cauchy data $h$ on $\Omega^*$ such that in time $T>0$, the wave field $U_h$ reaches $\Omega_{q}$ and we can specify the wave front set as well as the mode of the singularity (i.e. whether the singularity is on $P$ or $S$ bicharacteristics) of the solution at $\Sigma_{q}$.
We then add a suitable control (which is semiexplicit) to the Cauchy data, so that when propagated, it cancels the multiple scattering of the wave field after time $T>0$ so that we essentially generate a virtual elastic source at an interior point in $\Omega$, with an initial wavefront set given by a chosen codirection and a desired polarization.

With such specialized waves, we can recover $q$-interior travel times and $q$-interior lens data on $\Sigma_{q}^{+}$.
Using the local boundary rigidity result of Stefanov-Uhlmann-Vasy \cite{SUVRigidity}, we determine the wave-speeds $c_{\PS}$ on a neighbourhood $\mathcal{O}$ of $\Sigma_{q}^+$ in $\overline{\Omega}_{q}$. Here we use the fact that $\Omega$ has a strong convex foliation aligned with the interfaces. This much was already proven in \cite{CHKUElastic} and \cite{SUV2019transmission} so it remains to recover information on the density of mass.
If $\Sigma_q$ contains an interface, we measure the reflected amplitudes of such carefully constructed waves to determine the reflection operator $M_R$ on $\Sigma_{q}^{-}$.
Having the knowledge of $M_R$ on the boundary, we use the result of \cite[Theorem 1.1]{BHKU_1} to recover $\lambda$, $\mu$, $\rho$ and all their derivatives at $\Sigma_{q}^+$. With the specialized controls in the Cauchy data, we generate $P$-waves that are singular along a $P$-ray inside $\bar\Omega_q$ that connects two nearby points on $\Sigma^+_{q}$. We can then use the Weinstein symbols discussed in Section \ref{s: Weinstein calculus} to recover the lower order amplitudes of such waves. These are determined by $a_{\pm,k,\PS}$ restricted to a corresponding Lagrangian that we described earlier. \\

In the\textbf{ second part}, we consider an asymptotic expansion of the amplitude function $a_{\pm,k,\PS}$ on the $P$ bicharacteristics in $T^*\Omega_{q}$, starting and ending at $T^*\Sigma_{q}$.
We observe that along the $P$ bicharacteristics, the terms in the asymptotic expansion of the amplitude satisfy transport equations with the initial data prescribed at $\Sigma_{q}$.
Here, the boundary data consists of the solution and its Neumann derivative at the boundary $\Sigma_{q}$ of $\Omega_q$.

By varying the boundary data, we will recover the local geodesic ray transform of a 2-tensor $A_{\rho}$ given by the double derivatives of the density function $\rho$ in $\mathcal{O}\subset \overline{\Omega}_{q}$.
Using the injectivity result \cite{SUVlocaltensor} on the geodesic ray transform of 2-tensors, we recover the action of the \emph{Saint-Venant operator} on $A_{\rho}$.
In other words, we determine a 4th order elliptic PDE that $\rho$ satisfies in $\mathcal{O}\subset \overline{\Omega}_{q}$.
Using elliptic unique continuation results, we recover the density function $\rho$ in the neighbourhood $\mathcal{O}$ outside of the set $\mathcal{D}$, where the PDE is not elliptic.
Having $c_{\PS}$ and $\rho$, we obtain  the parameters $\lambda$, $\mu$ in $\mathcal{O}$.
We proceed by iteration and finally recover the parameters everywhere in $\Omega$. In the next subsection, we list the key ingredients and previous results that will be essential to the main proof.

\subsection{Summary of preliminary results}

\subsubsection{Local travel time tomography}
A key ingredient in the proof of uniqueness will be the following theorem proved by Stefanov, Uhlmann, and Vasy in \cite{UVLocalRay}.
\begin{theorem}\label{thm: UV local rigidity}
Choose a fixed metric $g_0$ on $\Omega$. Let $n= \text{dim}( \Omega) \geq 3$; let $c,\tilde c >0$ be smooth, and suppose $\p \Omega$ is convex with respect to both $g = c^{-2}g_0$ and $\tilde g = \tilde c^{-2}g_0$ near a fixed $p \in \p \Omega$. If $d_g(p_1,p_2)= d_{\tilde g}(p_1,p_2)$ for $p_1,p_2$ on $\p \Omega$ near $p$, then $c= \tilde c $ in $\Omega$ near $p$.
\end{theorem}

We write down a trivial corollary due to continuity of the distance function.
\begin{cor}\label{cor: lens rigidity from dense set of point}
Consider the same setup as in the above theorem. If $d_g(p_1,p_2)= d_{\tilde g}(p_1,p_2)$ for a dense set of points $p_1,p_2$ on some neighborhood of $p$ in $\p \Omega$, then $c= \tilde c $ in $\Omega$ near $p$.
\end{cor}
We need this since due to the multiple scattering in our setting, we will only be able to recover boundary travel times on a dense set of points and not a full neighborhood. We quote a similar result for the lens relation. Let $L$ denote the lens map.
\begin{cor}(\cite[Corollary]{SUVRigidity}) Let $\Omega, c, \tilde c$ be as above with $c = \tilde c$ on $\p \Omega$ near $p$. Let $L = \tilde L$ near $S_p \p \Omega$. Then $c = \tilde c$ in $\Omega$ near $p$.
\end{cor}
As before, the same corollary holds if we instead assume $L = \tilde L$ in a set that is dense inside some neighborhood of $S_p \p \Omega$.

\subsubsection{Recovery of wave speeds and density of mass across an interface from reflected waves}
In this section, we show that if all material parameters have already been recovered from $\mathcal F$ within a layer up to an interface $\Gamma_i$, then the parameters as well as their normal derivatives can be recovered infinitesimally across $\Gamma_i$; that is, $\p_\nu^J\rho^{(+)}, \p_\nu^J\lambda^{(+)}, \p_\nu ^J\mu^{(+)}$ can be recovered as well (see Corollary \ref{cor: all three parameters at Gamma from mathcal F}). We will do this by analyzing the amplitudes of waves reflected at $\Gamma_i$ from above.
Let us define the incoming and outgoing unit sphere bundle $\partial_{\pm}S^*\Omega_{q}:= \{ (x,\xi)\in \partial_{\pm}T^*\Omega_{q}; |\xi| = 1 \}$.
We start with a key proposition.
\begin{prop}\label{Neumann_tail_lemma}
Let us fix $(x,\xi) \in \partial_{+}S^*\Omega_{q}$. Let $V$ be a distribution, such that $WF(V) = \{(x,s\xi); s \in \R\setminus\{0\}\}$, supported outside $\Omega_{q}$ for some $q>0$. There exist a large enough $T>0$ and a Cauchy data $U_{\infty} \in \C$ supported in $\Omega^*$ such that
\begin{equation}
WF(\mathcal{F}_{T+s}U_{\infty}-\mathcal{F}_s V) = \emptyset, \quad \mbox{for }s\geq 0,\quad \mbox{in }\Omega_{q}.
\end{equation}
Moreover, one can arrange that the singularity of $U_{\infty}$ flows along the $P$ characteristics $\gamma_{P}^{\pm}$ outside $\Sigma_{q}$, i.e. $WF(U_{\infty}) \subset \gamma_{P}^{\pm}$ for time $t$ close enough to $T$.
Similarly one may take $WF(U_{\infty}) \subset \gamma_{S}^{\pm}$ for time $t$ close enough to $T$.
\end{prop}
\begin{proof}
	We will give a brief sketch of the proof, since it is similar to \cite[Proposition 5.3]{CHKUElastic}.
	First we prove the result for $(x,\xi) \in \mathcal{S} \subset T^*\Omega$, consists of $(x,\xi)$ such that all the \emph{bad bicharacteristics}\footnote{See \cite[definition C.1]{CHKUElastic}.} through $(x,\xi)$ are \emph{(+)-escapable}.
	We consider a $\PS$ purely transmitted bicharacteristic $\gamma_{\PS}$ through $(x,\xi)$ outside $\Omega_{\tau}$ such that $\gamma_{\PS}(T)=(x,\xi)$ and $\gamma_{\PS}(0) \in \partial\Theta$.
	Consider a Cauchy data $U_0$ supported in $\Theta\setminus\overline{\Omega}$ such that $WF(U_0) = \{ (\gamma(0),s\dot{\gamma}(0)) ; s\in \R_+ \}$. Using finite speed of propagation and the fact that the singularity flows along the bicharacteristics, one obtains $WF(\mathcal{F}_TU_0 -V) = \emptyset$.
	As in \cite[Appendix C]{CHKUElastic}, one can construct a tail denoted $K_{tail}$, that together with $U_0$, cancels the multiple scattering and $U_{\infty}:= U_0 + K_{tail}$ is the required Cauchy data. That construction relies on $M_T$ to be elliptic away from the glancing set on both sides of an interface, and this is independent of $\rho$, so the same proof holds.
	Finally, using a density argument (as in \cite[Lemma 5.9]{CHKUElastic}), we remove the restriction $(x,\xi) \in \mathcal{S}$.
\end{proof}

\begin{rem}
	The usefulness of the above proposition is in the layer stripping argument, where we recover the parameters outside $\Omega_{\tau}$ and probe the Cauchy data $U_{\infty}$ outside $\Omega$, which generates singularities along the bicharacteristics in $\Omega_{\tau}$. Furthermore, one can generate the flow of singularity along the $P$ and the $S$ bicharacteristics separately.
\end{rem}
\begin{rem}
	The time $T>0$ is determined by the time it takes for all the branches of the scattered waves to travel from $\Omega^*$ to $\Sigma_q$ and back to $\Theta^*$.
	Essentially, $T$ can be estimated by the $S$ distance between $\Omega^*$ and $x \in \Sigma_{q}$.
	For details see \cite[Remark 5.4]{CHKUElastic}.
\end{rem}

We now have an important lemma, taken from \cite[Lemma 5.6]{CHKUElastic} that allows us to recover the reflection operator from reflected waves measured outside $\Omega$.

\begin{lemma} \label{l: reflection psido from mathcal F}
Suppose that $\Sigma_q \subset \Gamma$ and $c_\PS = \tilde c_\PS, \rho = \tilde \rho$
 outside $\overline \Omega_\tau$. Assume $\mathcal{F} = \tilde {\mathcal F}$. Then \[
 M_R \equiv \tilde M_R \text{ on } T^*\Sigma_q^-.
 \]
\end{lemma}
\begin{proof}
 The proof is identical to that of \cite[Lemma 5.6]{CHKUElastic}. In that paper, the trivial density assumption was only needed to have $P = \tilde P$ outside $\overline\Omega_q$, which is also true in our case due to the assumptions given.
\end{proof}

We then have two important corollaries that follows from Theorem 1.5 in \cite{BHKU_1}.
\begin{cor} \label{cor: all three parameters at Gamma from mathcal F}
Suppose that $\Sigma_q \subset \Gamma$ and $c_\PS = \tilde c_\PS, \rho = \tilde \rho$
 outside $\overline \Omega_q$. Assume $\mathcal{F} = \tilde {\mathcal F}$. Then $\p_{\nu}^j c^{(+)}_{\PS} = \p_{\nu}^j\tilde c^{(+)}_{\PS}$ and $\p_{\nu}^j\rho^{(+)} = \p_{\nu}^j\tilde \rho^{(+)}$ on $\Gamma$ for all $j = 0, 1, 2, \dots$
\end{cor}

\begin{proof}
By Lemma \ref{l: reflection psido from mathcal F}, we recover $M_R$ over $\Gamma_-$ based on our assumptions. We then apply \cite[Theorem 1.5]{BHKU_1}.
\end{proof}

\begin{cor} \label{cor: all M_T at Gamma from mathcal F}
Suppose that $\Sigma_q \subset \Gamma$ and $c_\PS = \tilde c_\PS, \rho = \tilde \rho$
 outside $\overline \Omega_q$. Assume $\mathcal{F} = \tilde {\mathcal F}$. Then $M_T \equiv \tilde M_T$ and $T \equiv \tilde T$ at $\Gamma$.
\end{cor}

\subsubsection{Recovery of subsurface travel times and lens relations}
Next, we show that one can recover the subsurface lens relations when knowing the parameters outside the domain $\Omega_q$ and the outside measurement operator.

\begin{lemma}\label{l: lens relation recovery}
 Let $(x,\xi) \in \p T^* \Omega_q \cap S^*_+\Omega_q \cap \mathcal S$, and assume the extended convex foliation condition. If $\mathcal F = \tilde {\mathcal F}$ and $\lambda = \tilde \lambda$, $\mu = \tilde \mu$, $\rho = \tilde \rho$ outside $\Omega_q$, then $c_{\PS}$ and $\tilde c_{\PS}$ have identical $q$-interior lens relations w.r.t. $\Sigma_q$ in a neighborhood of $(x,\xi)$ within $T_{\Sigma_q}^*\Omega$.
\end{lemma}

\begin{proof}
The proof is identical to that of \cite[Lemma 5.6]{CHKUElastic}. In that paper, the trivial density assumption was only needed to have $P = \tilde P$ outside $\overline\Omega_q$, which is also true in our case due to the assumptions given. In fact, the lemma is true even without the density assumption by the proof of \cite[Theorem 10.2]{SUV2019transmission}.
\end{proof}

We can combine the above lemma with Theorem \ref{thm: UV local rigidity} to obtain the key corollary. First, let $d^q_{\PS}$ denote the $\PS$-distance function restricted to $\overline\Omega_q \times \overline\Omega_q$.
\begin{cor}\label{cor: local uniqueness}
With the assumptions in the above lemma, $d_{\PS}^q\big|_{\Sigma_q \times \Sigma_q} = \tilde d_{\PS}^q\big|_{\Sigma_q \times \Sigma_q}$ in some neighborhood of $x$, and $c_{\PS} = \tilde c_{\PS}$ in some neighborhood of $x$.
\end{cor}

\subsubsection{Unique continuation for elliptic operators}
Consider the differential inequality
\begin{equation}\label{e: inequality for UCT}
|\Delta^n u| \leq f(x,u, Du, \dots, D^k u)
\end{equation}
where $f$ is Lipschitz, $k = [3n/2]$.

We have a unique continuation result from \cite[Section 3]{Protter}.
\begin{theorem}[Protter Theorem]
\label{thm: unique continuation}
If $u$ satisfies \eqref{e: inequality for UCT} in a neighborhood $D$ of the origin and $u$ vanishes in any neighborhood of the origin (not necessarily $D$), then $u$ vanishes identically on $D$. In fact, the conclusion holds if
\[
e^{2/|x|^\alpha} u \to 0 \ as \ |x| \to 0
\]
for any positive $\alpha$.
\end{theorem}
In the main proof, we will eventually recover ``lower order'' amplitudes of elastic waves, and use them to show $\rho - \tilde \rho$ satisfies an elliptic, fourth order PDE and an inequality of the form \eqref{e: inequality for UCT}. We can then use Protter theorem to locally determine $\rho$.

Following the sketch of the proof at the start of Section \ref{s: recovery of the parameters}, we divide proof into two parts: \emph{The first part} concerns previous results on local travel time tomography that is used to recover the subsurface lens relation at each $\Sigma_{\tau}$ from the knowledge of the parameters in $\Omega_{\tau}^*$. We also analyze reflected amplitudes to recover the jumps in the material parameters and their derivatives across $\Sigma_\tau$ if there is an interface there. We then recover the amplitudes of $P$-waves along the bicharacteristic curves in a smooth neighbourhood of $\Sigma_{\tau}.$
\emph{Part 2} deals with the analysis of the amplitudes along the bicharacteristic curves in a smooth neighbourhood of $\Sigma_{\tau}$.
Using the lower order terms along with the principal part of the symbol $p(t,x,\tau,\xi)$, we obtain a transport equation for the amplitude, which in turn helps us to recover the density.

\subsection{\bf Proof of Theorem \ref{Main_Th_1} and absence of gauge freedom}\label{s: proof of the main theorem}

\begin{proof}[Proof of Theorem \ref{Main_Th_1}]
The proof is by contradiction. Suppose $c_P \neq \tilde c_P$ or $c_S \neq \tilde c_S$ or $\rho \neq \tilde \rho$, and let $f= |c_P - \tilde c_P|^2 + |c_S - \tilde c_S|^2 + |\rho - \tilde \rho|^2$. Now consider $S:= \Omega \cap \text{supp} f$, and take $q = \min_S \sig$: 
so $c_{P} = \tilde c_{P}$ and $c_S = \tilde c_S$ and $\rho = \tilde \rho$ above $\Omega_q$, but by compactness there is a point $x \in \Sigma_q \cap S$. The condition that $\sig^{-1}(\tilde q)$ has measure zero rules out the trivial case $q = \tilde q$.

{\bf First part:}
Let us now consider a small neighborhood of $x$, denoted $B_x$, and we consider the $\Sigma_q$-boundary distance function $d^q_{\PS}$ restricted to this neighborhood. Since the interfaces are not dense, and we assume convex foliation, we may choose $B_x$ small enough so that all $P$ and $S$ rays corresponding to rays staying completely inside $B_x$ do not reach an interface; i.e. even the mode converted rays do not reach an interface. This ensures that a $P$-wave that hits $B_x$, transmits a $P$ and $S$ wave, the $P$-wave returns to $\Sigma_q$ first before any other ray.

The proof of \cite[Theorem 10.2]{SUV2019transmission} (\cite{CHKUElastic} has a slightly different proof) in conjunction with Corollary \ref{cor: lens rigidity from dense set of point} or Theorem \ref{thm: UV local rigidity} shows $c_\PS = \tilde c_\PS$ in some neighborhood of $x$ inside $B_x$ which we keep denoting as $B_x$. Now let $(z_1, \zeta_1) \in \p T^*\Omega_q$ with $z_1 \in B_x$ such that the ray $\gamma=\gamma_{P,z_1,\zeta_1}$ starting at $(z_1,\zeta_1)$ at time $t=0$ remains inside $B_x$ until it hits $\Omega_q$ again at some point $(z_2,\zeta_2) = L_q(z_1,\zeta_1)$ at time $t_2$. Let $\Lambda_t$ be the Lagrangian associated to the $P$-bicharacteristic flow of time $t$, but restricted to $B_x$. To describe the Lagrangian $\Lambda_{t}$,
denote $\chi^P_t : T^*\RR^n \to T^*\RR^n$ the $P$-bicharacteristic flow by $t$ units of time. Then $\Lambda_{t} = \{ (\chi^P_{t}(y,\eta), y, \eta); (y,\eta) \in T^*B_x\}$.

Let $V \in \mathcal E'(\Omega)$ with $\WF(V) = \RR_+(z_1, \zeta_1)$. As shown in \cite[Proof of Theorem 10.2]{SUV2019transmission}, we may construct an outgoing $P$ wave $u, \tilde u$ in $\Omega_q$ so that near $z_1$, $u(t_1) \equiv V$ and $\tilde u (t_1) \equiv V$.
As in \cite{SUV2019transmission}, $\gamma$ does not hit an interface for either operator.
We now consider two cases, depending on whether $x$ is on an interface or not.
\medskip

\noindent\textbf{\boldmath Smooth case:} $x \notin \Gamma$.

 By the construction in \cite{SUV2019transmission} which uses Proposition \ref{Neumann_tail_lemma}, $u$ and $\tilde u$ are microlocal $P$-waves inside $B=[T,T+t_2] \times B_x$ and within this set are given by the forward propagator applied to $V$ with wavefront set in $\Sigma_P$:
\begin{equation}\label{e: propag applied to V for u}
u|_B \equiv \int e^{i \phi^+_P(t,x,\eta)} a^{\cdot, l}_{+,k,P}(t,x,\eta) \hat{V}_l(\eta) \ d\eta, 
\end{equation}
and likewise
\begin{equation}\label{e: propag applied to V for tilde u}
\tilde u|_B \equiv \int e^{i \tilde \phi^+_P(t,x,\eta)} \tilde a^{\cdot, l}_{+,k,P}(t,x,\eta) \hat{V}_l(\eta) \ d\eta.
\end{equation}
Since $\phi^+_P, \tilde \phi^+_P$ only depend on the wave speeds and we have recovered them inside $B_x$, these phase functions are equal within $B_x$.
 We would like to conclude that $(a_{+,k,P})_J - (\tilde a_{+,k,P})_J|_{\{(z_1,\zeta_1),(z_2,\zeta_2)\}} = 0$ for each $J = 0, -1, -2, \dots$ in the polyhomogeneous expansion of the symbols.

 Next, as shown in \cite{CHKUElastic}, $u \equiv \tilde u$ on $\Omega_q^c$ since $P = \tilde P$ outside $\Omega_q$ and using propagation of singularities.
 Thus,
\[
\int e^{i \phi^+_P(t,x,\eta)} (a^{\cdot, l}_{+,k,P}(t,x,\eta) - \tilde a^{\cdot, l}_{+,k,P}(t,x,\eta)) \hat{V}_l(\eta) \ d\eta
\in C^\infty
\]
over $\Omega_q^c$. Next, restrict to $t = t^* = t_2 +\epsilon$.
Let us denote the symbol $b(x,\eta) = a_{+,k,P}(T+t^*,x,\eta) - \tilde a_{+,k,P}(T+t^*,x,\eta)$ restricted to $B_x$ in the spatial variable. The leading order terms in the polyhomogeneous expansions of $a$ and $\tilde a$ only depends on the principal symbol of $P, \tilde P$, and are thus equal since we have already recovered the wavespeeds \cite{SUV2019transmission}. Thus,
$b \sim \sum_{J=-1,-2,\dots} b_J$ where $b_{-1} = (a^{\cdot, l}_{+,k,P}(t^*,x,\xi))_{-1} - (\tilde a^{\cdot, l}_{+,k,P}(t^*,x,\xi))_{-1}\in S^{-1}_{hom}$. Our goal is to conclude $b_{-1} = 0$ at $(z_2,\zeta_1)$. Restricted to $\bar \Omega_\tau^c$, we have
\[
BV:= \int e^{i \phi^+_P(t^*,x,\xi)} b(x,\xi) \hat{V}(\xi) \ d\xi
= f \in C^\infty(B_x \cap \bar \Omega^c_q)
\]
and $B$ is a Fourier integral operator associated to a canonical graph inside the class $\mathcal I^0(B_x \times B_x, \Lambda_{t^*})$. Note that $\Lambda_{t^*} = \{ (\chi^P_{t^*}(y,\eta), y, \eta); (y,\eta) \in T^*B_x\},$ and in particular, $(\bar z, \bar \zeta) = \chi^P_{t^*}(z_1,\zeta_1)$.
As shown in Appendix \ref{s: proofs of Weinstein symbol stuff}, $\theta(y):= x \circ \chi_{t^*}^P(y,\zeta_1)$ is a local diffeomorphism from some neighborhood of $z_1$ to a neighborhood of $\bar z$ that we denote by $W$. We can shrink neighborhoods so that $W \subset \bar\Omega^c_\tau.$ Let $N$ denote the order of the distribution $V$ (see definition \ref{d: order of distribution}) and let $u\in \mathcal C^\infty_c(W)$.
Denote $u_\tau(x) = \tau^{n/2} e^{-i\tau(x- \bar z) \cdot \bar \zeta}u(\sqrt \tau(x-\bar z)$.
Then corollary \ref{cor: symbol of FIO extraction with limits} shows that
\[
\liminf_{\tau \to \infty} \tau^{1-N} \langle BV, u_\tau \rangle
= b_{-1}^*(\bar z,  \zeta_2)J_{\theta^{-1}}(z_1)\liminf_{\tau \to \infty} \langle V \circ \theta^{-1}, u_\tau \rangle.
\]
Also, $\lim_{\tau \to \infty} \langle f, u_\tau \rangle = 0$ since $f$ is smooth by \cite[Corollary 2.2.2]{weinstein1976}. We have the freedom to choose $V$ so that this last limit is nonvanishing; we can even use the construction in \cite[Example 2.6]{Dencker-Polarization} so that the Fourier transform of $V$ belongs in a certain $0$'th order symbol class so that $N=0$.
  We conclude $b_{-1}(\bar z,  \zeta_1) = 0$. Repeating this argument for $\epsilon \to 0$ lets us conclude $b_{-1}(z_2, \zeta_1) = 0$ since $J_{\theta^{-1}}$ is nonvanishing. We may repeat this argument by adjusting the weight $\tau^{k-N}$ to show vanishing of the lower order terms of $b$ as well but we do not need it.

This can be done for any such $V$ and downward $(z_1,\zeta_1)$ near $T_x^*\p \Omega_q$ so we conclude by (\ref{prop: symbol of FIO applied to distribution})
\[
 a^{\cdot, l}_{+,k,P}(t_2,z_2,\zeta_1)
=\tilde a^{\cdot, l}_{+,k,P}(t_2,z_2,\zeta_1)
\text{ modulo }S^{-\infty}.
\]

\noindent\textbf{\boldmath Interface case: }$x \in \Gamma$.

 This case differs from the previous case since we want measurements of the amplitude on $\Sigma_+ \supset \Gamma_+$, while our assumptions only allow measurements on $\Gamma_-$. We are essentially treating $\Gamma_+$ (and not $\Gamma_-$ as the boundary for $\Omega_q$). Fortunately, the difference is given by the transmission operator which we can recover via Corollary \ref{cor: all M_T at Gamma from mathcal F}. Note that since we recovered the speeds, $\gamma = \tilde \gamma$.

 As shown in \cite{CHKUElastic}, our inductive assumptions, propogation of singularities and the foliation condition imply $u \equiv \tilde u$ in $\bar \Omega_q^c$ since $P = \tilde P$ in that region. Near $(T,z_1)$, by construction we have $u|_{\Gamma_+} = M_T u|_{\Gamma_-}$ and $\tilde u|_{\Gamma_+} = \tilde M_T \tilde u|_{\Gamma_-}$.
By Corollary \ref{cor: all three parameters at Gamma from mathcal F} and since $u|_{\Gamma_-} \equiv \tilde u|_{\Gamma_-}$, we conclude $u|_{\Gamma_+} \equiv \tilde u|_{\Gamma_+} \equiv V$ near $(T,z_1)$.

By construction in Proposition \ref{Neumann_tail_lemma}, $u$ and $\tilde u$ are microlocal $P$-waves inside $B = [T,T+t_2] \times B_x$ as in the previous case and within this set are given by the forward propagator applied to $V$ with wavefront set in $\Sigma_P$ just as in \eqref{e: propag applied to V for u} and \eqref{e: propag applied to V for tilde u} above. Near $(T+t_2,z_2)$, we have again have $u|_{\Gamma_-} \equiv M_T u_{\Gamma_+}$ and $\tilde u|_{\Gamma_-} \equiv \tilde M_T \tilde u_{\Gamma_+}$. Since $u \equiv \tilde u$ outside $\bar\Omega_q$ and using Corollary \ref{cor: all M_T at Gamma from mathcal F}, we conclude $u|_{\Gamma_+} \equiv \tilde u|_{\Gamma_+}$ near $(T+t_2,z_2)$ as well.
 Denote $\rho_{\Gamma_+}$ as the restriction to $\Gamma$ from below near $z_2$. As shown in \cite[Chapter 5]{DuisFIO}, if $\text{WF}(u)$ over $\Gamma_+$ contains no covectors tangential to $\Gamma$, then $\rho_{\Gamma_+}u$ is also the image of an FIO in $\mathcal{I}^0$ associated to a canonical graph $\Lambda \subset T^*(\Gamma_+ \times \RR_t) \times T^*B_x$. It can be described using $\chi_t^P$ from before. Then $\Lambda = \{ t(y,\eta), -|\eta|_P, d\rho_{\Gamma_+}\chi^P_{t(y,\eta)}(y,\eta), (y, \eta)\}$ where $t(y,\eta)$ is the time the $P$-ray from $(y, \eta)$ hits $\Gamma$ for the first time. It will be convenient to just define $\Phi(y,\eta) =  (t(y,\eta), -|\eta|_P, d\rho_{\Gamma_+}\chi^P_{t(y,\eta)})$. Let us introduce boundary normal coordinates for $\Gamma$ near $z_2$ with local coordinates $(x',x_n)$, where $\Gamma$ is given by $x_n = 0$. Using (\ref{e: propag applied to V for tilde u}) and (\ref{e: propag applied to V for u}), we then have near $(t_2,z_2)$
 \begin{equation}
 \rho_{\Gamma_+}u \equiv \int e^{i \phi^+_P(t,x',\eta)} a^{\cdot, l}_{+,k,P}(t,x',\eta) \hat{V}_l(\eta) \ d\eta
\end{equation}
and likewise
\begin{equation}
\rho_{\Gamma_+}\tilde u \equiv \int e^{i \tilde \phi^+_P(t,x',\eta)} \tilde a^{\cdot, l}_{+,k,P}(t,x',\eta) \hat{V}_l(\eta) \ d\eta.
\end{equation}
Analogous to the previous case, we
define $b(t,x',\eta) = a_{+,k,P}(t,x',\eta)
- \tilde a_{+,k,P}(t,x',\eta)$ restricted to a small neighborhood of $(t_2,z_2)$ within $\RR_t \times \Gamma_+$. As in the previous case, $b_0 = 0$ in the polyhomogeneous expansion of $b$. Restricted to $\bar \Omega_q^c$, we have as before
\[
BV:= \int e^{i \phi^+_P(t,x',\xi)} b(t,x',\xi) \hat{V}(\xi) \ d\xi
= f \in C^\infty(\RR_t \times \Gamma_+ \cap \RR_t \times \bar \Omega^c_q),
\]
where $B$ is a Fourier integral operator associated to the canonical graph $\Phi$ inside the class $\mathcal I^0(\RR_t \times \Gamma_+ \times B_x, \Lambda)$. To apply the argument in the previous case with Proposition \ref{prop: symbol of FIO applied to distribution}, all that is necessary is that $B$ is an FIO associated to a canonical graph so that $(t,x') \circ \Phi(y,\eta)$ is a local diffeomorphism, since we are away from the glancing set. Hence, using the test functions $u_\tau(x) = \tau^{n/2} e^{-i\tau[(x'-z_2)\cdot \zeta_2 - (t-t_2)|\zeta_1|_P] } u (\sqrt \tau (t-t_2),\sqrt \tau (x'-z_2)$, we apply the argument in the previous case to conclude $b_{-1}(t_2,z_2,\zeta_1) = 0$. Likewise, we can iterate to show the vanishing of the lower order terms as well.

  Hence, we are in the same situation as the smooth case above and the remaining argument proceeds as above to conclude $a_{\pm,k,P} = \tilde a_{\pm,k,P}$ on $\Lambda$ (the principal amplitudes may be taken as functions on $\Lambda$ via a diffeomorphsim \cite[Chapter 25]{Hormander_1}) modulo $S^{-\infty}$.

 {\bf Second part:}
Above, we have concluded that $a_{\pm,k,P} = \tilde a_{\pm,k,P}$ modulo $S^{-\infty}$ when restricted to the $P$-bicharacteristic flow Lagrangian. We calculate the amplitude $a_{\pm,k,P}$ on the $P$ bicharacteristic segment $\gamma$.

Recall the parametrization of the bicharacteristic curves as in \eqref{parametrization_bichar}. Let us evaluate the compatibility condition \eqref{Compatibility condition} for $J=0,-1$.
From a straight forward calculation we obtain (see \cite[Section 3]{B_density})
\begin{equation}\label{key_relation_1}
\begin{aligned}
\partial_{\tau,\xi}p(t,x,\tau,\xi)
=& 2\rho\left( -\tau I, \left[ c_S^2\xi_1 I + (c_P^2-c_S^2)(e_1 \circledS \xi)\right], \dots,\left[c_S^2\xi_1 I + (c_P^2-c_S^2)(e_3 \circledS \xi)	\right]\right)\\
\partial_{\xi_j\xi_k}p(t,x,\tau,\xi)
=& 2\rho \left[ \delta_{jk} c_S^2I + (c_P^2-c_S^2)(e_3 \circledS e_k) \right]\\
N \cdot \partial_{t,x}N
=& \frac{1}{2}\partial_{t,x}(N\cdot N) = 0.
\end{aligned}
\end{equation}

Let us recall from \eqref{asym_exp} that $\left(a_{\pm,k,P}^{\cdot,l}\right)_J = \left(h_{\pm,k,P}^{\cdot,l}\right)_J + \left(b_{\pm,k,P}^{l}\right)_J N$, for $J=0,-1,\dots$.
Note that $\left(h_{\pm,k,P}^{\cdot,l}\right)_0 = 0$ and let us denote $\left(b_{\pm,k,P}^{l}\right)_J = b_{J}$ for $J=0,-1$. Using the fact that $\left(a_{\pm,k,\PS}^{\cdot,l}\right)_1 = 0$, the compatibility condition \eqref{Compatibility condition} for $J=0$ reduces to
\begin{equation*}
\begin{aligned}
0 =& N\left[ B_P(b)_0\right]\\
=& 2i\rho \left[ -\tau\p_t + c_P^2(\xi\cdot\nabla_x) + \frac{1}{2}(c_P^2-c_S^2)|\xi|(\nabla_x\cdot N) + ic_P^2 \xi\cdot\nabla_x(\log \rho c_P^2)  \right] (b)_0\\
&-i\rho \left[c_P^2 N(\nabla\otimes \xi)N + c_P^2\frac{d}{ds}(\log c_P)|\xi| - c_S^2(\nabla_x\cdot\xi) - (c_P^2-c_S^2)N(\nabla_x\otimes\xi)N \right] (b)_0.
\end{aligned}
\end{equation*}
The above equation, combined with \eqref{parametrization_bichar}, reduces to a transport equation
\begin{equation*}
\frac{d}{ds}(b)_0 = -\frac{1}{2}\left[ \frac{d}{ds} \log (\rho c_P) + (\nabla_x\cdot N) \right] (b)_0.
\end{equation*}
For $s>T_0$, we solve the transport equation above and get (see \cite[Equation 16]{B_density})
\begin{equation}\label{b_0}
(b)_0(s) = (b)_0(0) \sqrt\frac{(\rho c_P)|_{s=0}}{(\rho c_P)(s)} \exp \left(-\frac{1}{2} \int_{0}^s (\nabla_x\cdot N)(r)\,dr \right), \quad\mbox{and}\quad
(a)_0 = N \otimes (b)_0.
\end{equation}

For $J=-1$, the compatibility condition \eqref{Compatibility condition} reads
\begin{equation*}
N[B_P (a)_{-1} + C_P (a)_0] = 0.
\end{equation*}
Using a similar calculation as for $J=0$ we get
\begin{equation}\label{key}
\frac{d}{ds} (a)_{-1} + \frac{1}{2}\left[ \frac{d}{ds}(\log \rho c_P) + (\nabla_x\cdot N)\right] (a)_{-1} = G
= \frac{1}{2i\rho c_P^2 |\xi|}\left[ NB_{P}(h^{\cdot,l}_{\pm,k,P})_{-1} + N C_P (a)_0 \right].
\end{equation}
Note that one can determine
\begin{equation*}
(h^{\cdot,l}_{\pm,k,P})_{-1} = \frac{-1}{\rho(c_P^2-c_S^2)|\xi|^2}B_P(N\otimes(a)_0),
\end{equation*}
from \eqref{Eqn for a-J}.
Thus, we can solve the above transport equation for $(a)_{-1}$ and get
\begin{equation}\label{a_-1}
g (a)_{-1} = C + \int_{\gamma} g G,
\end{equation}
where $C$ is a constant and
\begin{equation*}
g(s) = \frac{\sqrt{(\rho c_P)(s)}}{(a)_0(T_0)\sqrt{(\rho c_P)(T_0)}}\exp \left( \int_{0}^{s}\nabla_x\cdot N(r)\,dr \right).
\end{equation*}

Now, note that $U$, $M_T$, and $\p_{\nu}^j\lambda$, $\p_{\nu}^j\mu$, $\p_{\nu}^j\rho$, for $j=0,1,\dots$ are known at $\Sigma_{\tau}^\pm$, thus we know $(a)_{-1}|_{s=T_0}$.
Let there be two sets of parameters $(\lambda,\mu,\rho)$ and $(\tilde{\lambda},\tilde{\mu},\tilde{\rho})$ in $\Omega$ as assumed in the statement of Theorem \ref{Main_Th_1}. Let $(\tilde{a})_{-1}$ be the same quantity as $(a)_1$ corresponding to the parameters $(\tilde{\lambda},\tilde{\mu},\tilde{\rho})$.
Therefore, from the above analysis we get $(a)_1(0) = (\tilde{a})_1(0)$.
Since we already have that the wave-speeds $c_{\PS}$ are same in a neighbourhood of $\Sigma_{\tau}$, hence the geodesics are also same near $\Sigma_{\tau}$.
Since the projection of $\zeta$ in $\Omega$ is a geodesic $\gamma$, therefore, from \eqref{a_-1} we obtain
\begin{equation}\label{Ray_transform_1}
\int_{\gamma} N\cdot \left(\frac{A -\tilde{A}}{c_P}\right)N\,ds = 0,
\end{equation}
where $A$ and $\tilde{A}$ are exactly same as the matrices $A_1$, $A_2$ derived in \cite[Equation 21]{B_density}.
Here $A$, $\tilde{A}$ are 2-tensors and using the local inversion result \cite{SUVlocaltensor} we obtain the kernel of the geodesic ray transform \eqref{Ray_transform_1}, given by the \emph{Saint-Venant operator} (see \cite{B_density}).
The \emph{Saint-Venant operator} applied on the 2-tensor $\frac{1}{c_P}(A-\tilde{A})$, results in to a symmetric 4-tensor.

Using the exact same calculations of the symmetric 4-tensor as in \cite[Section 4]{B_density} we finally obtain the following 4-th order PDE
\begin{equation}\label{PDE}
\frac{(c_P^2-c_S^2)(c_P^2-4c_S^2)}{c_P^4 - 5c_P^2c_S^2 +8c_S^4}\Delta^2 \log\left(\frac{\rho}{\tilde{\rho}}\right) - \Delta\left(\nabla_x \log(\rho\tilde\rho)\cdot\nabla_x\log\left(\frac{\rho}{\tilde{\rho}}\right)\right)=0,
\end{equation}
in a neighbourhood of $\Sigma_{q}$ in $\overline{\Omega}_{q}$.
Since we know the derivatives of the parameters on $\Sigma_{q}^+$ by Corollary \ref{cor: all three parameters at Gamma from mathcal F}, thus we can smoothly extend them in a neighbourhood of $\Sigma_{q}$ in $\Omega_{q}^*$.
If we exclude the set $\mathcal{D}={c_P = 2c_S}$ then the equation \eqref{PDE} is elliptic and we can use strong unique continuation to prove $\left(\log\rho - \log \tilde{\rho}\right) = 0$ in the neighbourhood of $\Sigma_{q}$.

 In particular, we have shown
 \[
 \beta^-:= \log \rho - \log \tilde \rho = 0 \text{ in }B_x \cap \Omega_q.
  \]
  By Corollary \ref{cor: all three parameters at Gamma from mathcal F}, $\beta^-$ vanishes to infinite order at $\p \Omega_q$, and so we may extend all three parameters smoothly outside $\Omega_q$ with $\beta_-$ extended by $0$. Let $B_{ext}$ be a neighborhood of $x$ such that $B_{ext} \cap \Omega_q = B$. With the extended parameters, $\beta^-$ is extended as well, continues to satisfy \eqref{PDE}, and in particular, over $B_{ext}$.

  Now, if $\digamma = \frac{(c_P^2-c_S^2)(c_P^2-4c_S^2)}{c_P^4-5c_P^2c_S^2+8c_S^4}$ is bounded on $B_{ext}$, then \eqref{PDE} shows that
  \[
  |\Delta^2(\beta^-)| \leq C g(x, D\beta^-, D^2 \beta^-, D^3 \beta^-)
  \]
  so if we replace $\beta^-$ by $\beta^-(x-x_0)$ for and $x_0$ in $B_{ext} \setminus \bar \Omega_q$, then $\beta^-$ indeed satisfies an inequality of the form \eqref{e: inequality for UCT} as well as the hypothesis of Theorem \ref{thm: unique continuation}. Hence, $\beta^- = 0$ on $B_{ext}$.

Thus, we can recover $\rho$ in a neighbourhood of $\Sigma_{q}$, excluding the set $\mathcal{D}$. Since we already have recovered $c_{\PS}$, so, we can recover all the parameters near $\Sigma_{q}$. Thus, $S = \emptyset$ which proves the theorem.

\end{proof}

\section{Conclusion}
This paper implies that under certain geometric conditions, a piecewise smooth coefficient of a hyperbolic partial differential operator that is not in its principal symbol may be uniquely recovered. The essential ingredient is having a microlocal parametrix to represent solutions via an FIO in order to recover travel times and ``lower order'' polarization terms. This reduces the problem to a local tensor tomography problem. Even the scattering control construction may be generalized to whenever the transmission operator is an elliptic operator, which depends on the original operator and the transmission conditions (see \cite{CHKUElastic} for the construction).\\
\indent However, this does not automatically lead to a unique recovery since the tensor tomography problem always has a gauge freedom. Thus, determining whether the gauge freedom may be eliminated will have to be done on a case by case basis that will be unique to the partial differential operator at hand.

\appendix
\section{Proofs of statements in Section \ref{s: Weinstein calculus}}\label{s: proofs of Weinstein symbol stuff}

In this section, we provide the proofs of the statements made in Section \ref{s: Weinstein calculus}. First we prove Lemma \ref{l:PSIDO applied to u_tau} on the principal symbol of a PsiDO applied to a distribution.
\begin{proof}[Proof of Lemma \ref{l:PSIDO applied to u_tau}]
Let $P^*$ be the formal distributional adjoint of $P$ with symbol denoted $p^*$.
Then
\begin{align*}
\langle (P\g)^\tau_\phi, u \rangle &=
\langle P\g, e^{-i\tau x \cdot \xi} u (\sqrt\tau x) \rangle \tau^{n/2},
\\
&=
\langle \g, P^*e^{-i\tau x \cdot \xi} u (\sqrt\tau x) \rangle \tau^{n/2}
\\
P^*e^{-i\tau x \cdot \xi} u (\sqrt\tau x)
&= c_n \int e^{i(x-y)\cdot \eta - i \tau y \cdot \xi}
p^*(x,\eta) u (y \sqrt \tau) dy d\eta
\\
&= c_n \tau^{-n/2}
\int e^{ix\cdot \eta -i\frac{y}{\sqrt \tau} \cdot(\eta + \tau \xi)}
p^*(x,\eta) u (y) dy d\eta
\\
&= c_n \tau^{-n/2}
\int e^{ix\cdot \eta -iy \cdot(\eta/\sqrt \tau + \sqrt\tau \xi)}
p^*(x,\eta) u (y) dy d\eta,
\end{align*}
where we changed variables $\sqrt \tau y = y$ in the last line.
Change variables
\[
\tilde \eta = \eta/\sqrt \tau + \sqrt \tau \xi
\]
so that
\[
\eta = \tilde \eta \sqrt \tau - \tau \xi
\]
and $d\eta = \tau^{n/2} d\tilde \eta$. Substituting gives
\begin{align*}
&=
c_n  \int e^{ix\cdot(\tilde \eta \sqrt \tau - \tau \xi)-iy \cdot \tilde \eta}
p^*(x,\sqrt \tau \tilde \eta - \tau \xi) u (y ) dy d\tilde \eta
\\
&=
c_n e^{-i \tau x\cdot \xi} \int e^{i(x \sqrt \tau - y) \cdot \tilde \eta}
p^*(x,\sqrt \tau \tilde \eta - \tau \xi) u (y ) dy d\tilde \eta
\\
&:= e^{-i \tau x\cdot \xi} v(\sqrt \tau x),
\end{align*}
where $v^\tau(x) = c_n\int e^{i(x  - y) \cdot \tilde \eta}
p^*(\frac{x}{\sqrt \tau},\sqrt \tau \tilde \eta - \tau \xi) u (y ) dy d\tilde \eta$
which is a PsiDO applied to $u$. Hence, we have shown that
\[
\langle (P\g)^\tau_\phi, u \rangle
= \langle (\g)^\tau_\phi, v^\tau(x) \rangle,
\]
i.e. the family $(\g)^\tau_\phi$ applied to the ``test function'' $v^\tau(x)$ which also depends on $\tau$.

Now suppose $P$ is a classical PsiDO so $p^*$ is homogeneous of degree $m$. Let us do the first order Taylor expansion of $p^*(\frac{x}{\sqrt \tau},\sqrt \tau \tilde \eta - \tau \xi)$ around $x =0$ and $\tilde \eta = 0$:
\begin{equation}
p^*(\frac{x}{\sqrt \tau},\sqrt \tau \tilde \eta - \tau \xi)
= p^*(0, -\tau \xi) + \tau^{-1/2} x \cdot \p_x p^*(0, -\tau \xi)
+ \tau^{1/2} \tilde \eta \cdot  \p_{\eta} p^*(0, -\tau \xi).
\end{equation}
 However, note that $\p_\eta p^* \in \p_\eta p^*_m + S_{hom}^{m-2}
 = \tau^{m-1} \p_{\eta} p_m^*(0, -\xi) + S_{hom}^{m-2}$.
 Thus,
 \[
 p^*(\frac{x}{\sqrt \tau},\sqrt \tau \tilde \eta - \tau \xi)
= p^*(0, -\tau \xi) + O(\tau^{m-1/2}),
 \]
 so we obtain
 \[
 v^\tau(x) = p^*(0, -\tau \xi)u(x) + O(\tau^{m-1/2}).
 \]
We have now shown
 \[
 \langle (P\g)^\tau_\phi, u \rangle
 = \langle p^*(0, -\tau \xi) (\g)^\tau_\phi, u(x) \rangle + O(\tau^{N+m-1/2}),
 \]
 which gives use the desired result on the principal symbol.\\

For the latter statement in the lemma, we have
\begin{align*}
Pe^{-i\tau (x-x_0) \cdot \xi} u (\sqrt\tau (y-x_0))
&= c_n \int e^{i(x-y)\cdot \eta - i \tau (y-x_0) \cdot \xi}
p(x,\eta) u (y \sqrt \tau) dy d\eta
\\
&= c_n \tau^{-n/2}
\int e^{i((x-x_0)-y/\sqrt \tau)\cdot \eta -i\sqrt \tau y \cdot \xi}
p(x,\eta) u (y) dy d\eta
\\
&= c_n \tau^{-n/2}
\int e^{i\tilde x\cdot \eta -i\frac{y}{\sqrt \tau} \cdot(\eta + \tau \xi)}
p(x,\eta) u (y) dy d\eta
\\
&= c_n \tau^{-n/2}
\int e^{i\tilde x\cdot \eta -iy \cdot(\eta/\sqrt \tau + \sqrt\tau \xi)}
p(x,\eta) u (y) dy d\eta
\end{align*}
where we changed variables $\sqrt \tau (y-x_0) = y$ in the second line and we denote $\tilde x = x-x_0$.
Change variables
\[
\tilde \eta = \eta/\sqrt \tau + \sqrt \tau \xi
\]
so
\[
\eta = \tilde \eta \sqrt \tau - \tau \xi,
\]
and $d\eta = \tau^{n/2} d\tilde \eta$.
Substituting gives
\begin{align*}
&=
c_n  \int e^{i\tilde x\cdot(\tilde \eta \sqrt \tau - \tau \xi)-iy \cdot \tilde \eta}
p(x,\sqrt \tau \tilde \eta - \tau \xi) u (y ) dy d\tilde \eta
\\
&=
c_n e^{-i \tau \tilde x\cdot \xi} \int e^{i(\tilde x \sqrt \tau - y) \cdot \tilde \eta}
p(x,\sqrt \tau \tilde \eta - \tau \xi) u (y ) dy d\tilde \eta
\\
&:= e^{-i \tau \tilde x\cdot \xi} v(\sqrt \tau \tilde x),
\end{align*}
where $v^\tau(x) = c_n\int e^{i(x  - y) \cdot \tilde \eta}
p(\frac{x}{\sqrt \tau}+x_0,\sqrt \tau \tilde \eta - \tau \xi) u (y ) dy d\tilde \eta$ and can be viewed
as a parameter dependent PsiDO applied to $u$.

Now suppose $P$ is a classical PsiDO so $p$ is homogeneous of degree $m$. Let us do the first order Taylor expansion of $p(\frac{x}{\sqrt \tau}+x_0,\sqrt \tau \tilde \eta - \tau \xi)$ around $x =0$ and $\tilde \eta = 0$.
\begin{equation}
p(\frac{x}{\sqrt \tau}+x_0,\sqrt \tau \tilde \eta - \tau \xi)
= p(x_0, -\tau \xi) + \tau^{-1/2} x \cdot \p_x p(x_0, -\tau \xi)
+ \tau^{1/2} \tilde \eta \cdot  \p_{\eta} p(x_0, -\tau \xi) + O(\tau^{-1})
\end{equation}
 However, note that $\p_\eta p = \p_\eta p_m \text{ mod }S^{m-2}
 = \tau^{m-1} \p_{\eta} p_m(0, -\xi) \text{ mod } S^{m-2}$. Since $u$ is compactly supported, we also have
 \begin{equation}\label{e: the x term bounded by tau minus half}
 x \cdot \p_x p(x_0,-\tau \xi) u(\sqrt \tau x) = O( \tau^{m-1/2})
 \end{equation}
 since $|x| \lesssim \tau^{-1/2}$ on the support of $u$. The rigorous details follow the proof in \cite[Theorem 2.1.2]{weinstein1976} with appropriate cutoff functions.

 Thus,
 \[
 p(\frac{x}{\sqrt \tau}+x_0,\sqrt \tau \tilde \eta - \tau \xi)
= p(x_0, -\tau \xi) + O(\tau^{m-1/2}),
 \]
 so we obtain
 \[
 v^\tau(x) = p(x_0, -\tau \xi)u(x) + O(\tau^{m-1/2}).
 \]
Replacing $x$ by $\sqrt \tau \tilde x$ and using \eqref{e: the x term bounded by tau minus half} gives us exactly
$ Pu_\tau = p(x_0,-\tau \xi) u_\tau + O(\tau^{m-1/2}).$

\end{proof}

We now prove Proposition \ref{p: pullback of distribution symbol} on the symbol of a pullback of a distribution by a diffeomorphism.

\begin{proof}[Proof of Proposition \ref{p: pullback of distribution symbol}]
This follows easily from the proof of \cite[Proposition 1.4.1]{weinstein1976} once we make several observations. First, we have by construction
\begin{align*}
&\tau^{-n/2}\Big(\langle (\g \circ \theta)^\tau_{\phi \circ \theta}, u \rangle
 - \langle \g ^\tau_\phi \circ d_{x_0}\theta, u \rangle \Big )\\
&= \tau^{-n/2}\Big(\langle \g ^\tau_\phi , J_{\theta^{-1}}(y)u(\sqrt \tau (\theta^{-1}(y)-x_0)) \rangle
- \langle \g  , e^{-i\tau \phi(y)}J_{\theta^{-1}}(y_0)u(\sqrt \tau(d_{x_0}\theta)^{-1}(y-y_0)) \rangle \Big)
\\
&=\tau^{-n/2} \langle \g , e^{-i \tau\phi(y)}\Big[J_{\theta^{-1}}(y)u(\sqrt \tau (\theta^{-1}(y)-x_0)) - J_{\theta^{-1}}(y_0)u(\sqrt \tau(d_{x_0}\theta)^{-1}(y-y_0))\Big] \rangle
\\
&= \tau^{-n/2}\langle \g, e^{-i \tau\phi(y)}v^\tau(\sqrt \tau (y - y_0)) \rangle,
\end{align*}
where
\[
v^\tau(y) = J_{\theta^{-1}}(y/\sqrt \tau + y_0) u(\sqrt \tau (\theta^{-1}(y/\sqrt \tau + y_0)-x_0) )
- J_{\theta^{-1}}(y_0)u((d_{x_0}\theta)^{-1}y).
\]
Denote
\[
A := d_{y_0}\theta^{-1} = (d_{x_0}\theta)^{-1}.
\]
Using Taylor series, we have
\beq \nonumber
\theta^{-1}(Y+y_0) = \theta^{-1}(y_0) + AY
+ \sum_{jkl} b_{jkl}(Y)Y_k Y_l
\eeq
where $b_{jkl}(Y)$ is a smooth $n \times n$ matrix function.

Thus,
\beq \label{e: theta inv Taylor equation}
 \stau( \theta^{-1}(y/\stau) - x_0)
  = Ay + \tau^{-1/2} \sum_{jkl} b_{jkl}(y/\stau)y_k y_l.
\eeq
A Taylor expansion for $u$ is
\[
u(z+h) = u(z) + a(z, h) \cdot h
\]
where $a(z, h)$ is a smooth $n-$vector function of $z$ and $h$, and similarly
\[
J_{\theta^{-1}(y)}(y/\stau + y_0) = |A| + \tau^{-1/2} C(y/\stau) \cdot y
\]
where $C(Y)$ is a smooth $n$-vector function.
Combining these equations with \eqref{e: theta inv Taylor equation} we obtain
\begin{multline}
v^\tau(y)
=
J_{\theta^{-1}(y)}\left(y/\stau + y_0\right)u\big(Ay + \tau^{-1/2} \sum_{jkl}a_j(y,  b_{jkl}(y/\stau)y_k y_l\big)
- |A|u(Ay)
\\
= \tau^{-1/2}\left[\sum_{jkl}a\Big( Ay,\stau( \theta^{-1}(y/\stau) - x_0)\Big)  b_{jkl}\left(\frac{y}{\tau}\right)y_k y_l\right]J_{\theta^{-1}(y)}(y/\stau + y_0)
\\
+ \tau^{-1/2} u(Ay)C(y/\stau) \cdot y .
\end{multline}
Hence, the above equation is analogous to \cite[Equation (1.4.9)]{weinstein1976} and the rest of the proof follows closely to \cite[Proof of Proposition 1.4.1]{weinstein1976}.
\end{proof}

We know prove Proposition \ref{prop: symbol of FIO applied to distribution} regarding the symbol of an FIO applied to a distribution.

\begin{proof}[Proof of Proposition \ref{prop: symbol of FIO applied to distribution}]
It will ease notation to use $\tilde A$ in place of $A$ in the statement of the proposition.
We start with a test function $u \in \mathcal D(X)$ so
\begin{align}\label{e: A^t applied to test function}
\langle (\tilde A\g)^\tau_{x_0,\xi_0}, u \rangle &=
\langle \tilde A\g, e^{-i\tau (x-x_0) \cdot \xi_0} u (\sqrt\tau (x-x_0)) \rangle \tau^{n/2}
\\
&=\langle \g, \tilde A^te^{-i\tau (x-x_0) \cdot \xi_0} u (\sqrt\tau (x-x_0)) \rangle \tau^{n/2}
\end{align}
Set $A = \tilde A$ which is also an FIO associated to a canonical graph.
An FIO associated to a canonical graph has the form
\[
Au(y) = \int e^{iS(y,\xi)} a(y, \xi) \hat u(\xi) \ d\xi.
\]
Denote
$u_\tau = e^{i(x-x_0)\cdot \xi_0}u(\sqrt \tau(x-x_0))$,
and we have

Now, we view $u$ as a wave packet centered at $x_0$ around $\xi_0$.
We do the Taylor expansion
\[
S(y, \xi) = S(y,\xi_0) + S'_\xi(y,\xi_0) \cdot( \xi-\xi_0) + q(y,\xi-\xi_0)
\]
where $q$ is quadratic in $\xi-\xi_0$. By homogeneity, $S(y,\xi_0) - S'_{\xi}(y,\xi_0) \cdot \xi_0 = 0 $. So
\[
S(y, \xi) = S'_\xi(y,\xi_0) \cdot \xi + q(y,\xi-\xi_0)
\]
Thus,
\[
Au(y) = \int e^{iS'_{\xi}(y,\xi_0) \cdot \xi} e^{iq(y,\xi)} a(y, \xi) \hat u(\xi) \ d\xi.
\]
Next, $\chi$ being a canonical graph implies that the map
\[
y \mapsto S'_\xi(y,\xi_0) = x(y,\xi_0):=T(y)
\]
is a diffeomorphism near $y_0 = \pi_1\chi (x_0,\xi_0)$, where $\pi_1$ is the projection to the base manifold. Applying a pullback via the inverse of the above map gives
\[
Au(y(x)) = \int e^{ix \cdot \xi} e^{iq(y(x),\xi)} a(y(x), \xi) \hat u(\xi) \ d\xi
= B(x,D)u
\]
where $B$ is a PsiDO with symbol $b(x,\xi)=e^{iq(y(x).\xi)}a(y(x),\xi)$. Technically, $b$ is not a symbol but it will be after introducing cutoff functions as in the proof of \cite[Theorem 3.2.5]{weinstein1976}. Denote
$u_\tau = e^{i(x-x_0)\cdot \xi_0}u(\sqrt \tau(x-x_0))$ and
\[
Au_\tau(y(x)) =  \int e^{i x \cdot \xi} e^{i q(y(x),\xi)} a(y(x),  \xi) \hat u_\tau(\tau \xi) \ d\xi.
\]
 As in the proof of \cite[Theorem 3.2.5]{weinstein1976}, we divide the above integral into two pieces with a cutoff function $\mathcal K(\Psi), \Psi = \xi - \xi_0$ which is $1$ for $|\Psi| \leq 1/2$ and supported in the unit ball. Thus, without loss of generality, we can replace $b$ with $\mathcal K(\tau^{1/2}\Psi) b$, with the piece containing $1-\mathcal K(\tau^{1/2}\Psi)$ being $O(\tau^{-\infty})$ via the proof of \cite[Theorem 3.2.5]{weinstein1976} and can be ignored. Thus,
\[
 Au_\tau(y) = B(x,\tau D)u_{\tau}|_{x = x(y)}
\]
where $ b(x,\xi)= \mathcal K(\tau^{1/2}\Psi) e^{iq(y(x),\xi-\xi_0)}a(y(x),\xi)$ is now a symbol and $B(x, D)$ is a (non-classical) PsiDO with homogeneous principal symbol $a_\mu(y(x),\xi)$.
This shows that $Au_\tau$ is merely an application of a pseudodifferential operator followed by a diffeomorphism. Due to the cutoff $\mathcal K$, the proof of Lemma \ref{l:PSIDO applied to u_tau} goes through. We showed that to leading order

\[
B(x,D)u_\tau = b_\mu(x_0,\tau \xi_0) u_\tau(x) + O(\tau^{-1/2})
\]
Thus, we get
$A u_\tau = e^{iq(y_0,\xi_0)}a_\mu(y_0,\xi_0)u_\tau(x(y)) + O(\tau^{-1/2})$.

Next observe that
\[
u_\tau(x(y)) = e^{i\tau(x(y)-x_0)\cdot\xi_0}u(\sqrt \tau (x(y)-x_0))     \tau^{n/2}
\]
and note that the phase $\phi(y) = (x(y) - x_0)\cdot \xi_0 = (S'_\xi(y,\xi_0) - x_0)\cdot \xi_0 = S(y,\xi_0) - x_0\cdot \xi_0$ satisfies the conditions in \cite{weinstein1976} for the phase, and the so the phase invariance results there apply. We have $\phi(y_0) = 0$ and $\phi'(y_0) = S'_y(y_0,\xi_0) = \eta_0$, so \cite[Proposition 1.2.1]{weinstein1976} shows we may replace $\phi$ by the phase $(y-y_0)\cdot \eta_0$ without changing the principal symbol.

Let $\theta$ be defined such that $\theta^{-1}(y) = T(y)$. Thus, using Proposition \ref{p: pullback of distribution symbol} we have
 \begin{align}
 \langle g, u_\tau(x(y) \rangle
& =  \langle J^{-1}_{\theta^{-1}}g, J_{\theta^{-1}}u_\tau(x(y) \rangle
\\
 &= \langle (J_{\theta}g\circ \theta)^\tau_{\phi \circ \theta}, u(x) \rangle
 = \langle (J_{\theta}g)^\tau_{\phi} \circ d_{x_0}\theta, u(x) \rangle + O(\tau^{N-1/2}).
 \end{align}
 We also compute
 \begin{equation}
 \sigma_{\chi(y_0,\eta_0)} (J_\theta g \circ \theta) = \sigma_{y_0,\eta_0}(\theta)[ \sigma_{y_0,\eta_0}(J_\theta g)]
 = J_\theta(y_0)\sigma_{y_0,\eta_0} (\theta) [\sigma_{y_0,\eta_0}(g)]
 \end{equation}
 using that multiplication by $J_\theta$ can be viewed as an application of a $0$'th order classical pseudodifferential operator.
 Piecing everything together gives us
 \[\sigma_{x_0,\xi_0}(\tilde A g)
 = a(y_0,\tau \xi_0)J_{T^{-1}}(y_0)\sigma_{y_0,\eta_0} (T^{-1}) [\sigma_{y_0,\eta_0}(g)].
 \]
 
\end{proof}

\section{Elastic wave parametrix with scattering}
\label{a: elastic parametrix}

In this section, we summarize the microlocal parametrix construction for the system \eqref{elastic_Neumann} with transmission conditions that was used to prove Proposition \ref{Neumann_tail_lemma} to recover the wave speeds, which was already done in \cite{CHKUElastic}.

\subsection{Cauchy data and propagator}
Recall the space $\C$ in \eqref{Cauchy_data_set}, where we define the Cauchy data.
Let $\Psi=(\psi_0,\psi_1) \in \mathbf{C}$ be some Cauchy data. Observe that for $\Psi \in \C$, there exist unique solution $U=U_{\Psi} \in C(\R,H^1(\R^3))$ of the initial value problem \eqref{elastic_Neumann}.
We define the \emph{propagator}
\begin{equation*}
F: \C \to C(\R,H^1(\R^3)) \quad \mbox{ as }\quad
F(\Psi) := U(t,x), \quad \mbox{in }\R\times\R^3.
\end{equation*}
The Cauchy propagation operator is defined as
\begin{equation}\label{propagation_op}
\mathcal{F}_t : \C \mapsto \C
\quad\mbox{ as }\quad
\mathcal{F}_t \Psi := r_{t}\circ F(\Psi) = \left(U(t,\cdot),\p_tU(t,\cdot)\right),
\end{equation}
where $U=U_{\Psi}$ and $r_t$ is the restriction at time $t$.

Let us define $\lambda_j$, $\mu_j$, $\rho_j$ to be the smooth extensions of the parameters $(\lambda,\mu,\rho)|_{\Omega_j}$ outside $\Omega_j$ so that the solution operators $E_0$ and $E_1$ may be defined for each set of such parameters. We define the Cauchy to solution operators $\mathbf{J}_{\C\to\mathbf{S}}$ and $\mathbf{J}_{\C\to\mathbf{S}_{+}}$ as in \cite[Section 4.1]{CHKUElastic}. Loosely speaking, $\mathbf{J}_{\C\to\mathbf{S}}$ maps the Cauchy data from $\C$ to the unique solution $U_{\Psi} \in C\left(\R,H^{1}(\Omega)\right)$ and $\mathbf{J}_{\C\to\mathbf{S}_+}U$ in same as $\mathbf{J}_{\C\to\mathbf{S}}U$ but only propagates forward in time.
We also borrow the the following operators from \cite{CHKUElastic}, given as
\begin{equation*}
\begin{aligned}
\mathbf{J}_{\C \to \partial}	&: \mbox{ maps Cauchy data from }\C \mbox{ to the boundary }C(\R,H^{1/2}(\partial\Omega_j)),\\
\mathbf{J}_{\C \to \partial_+}	&: \mbox{ maps Cauchy data from }\C \mbox{ to the boundary }C(\R,H^{1/2}(\partial\Omega_{j})),\\
\mathbf{J}_{\partial \to \mathbf{S}}&: \mbox{ maps boundary data }C(\R,H^{1/2}(\partial\Omega_{j+1})) \mbox{ to the solution } C(\R,H^{1}(\Omega)),\\
\mathbf{J}_{\partial\to\partial}	&: \mbox{ maps boundary data }C(\R,H^{1/2}(\partial\Omega_{j})) \mbox{ to a different boundary }C(\R,H^{1/2}(\partial\Omega_{j+1})).
\end{aligned}
\end{equation*}


\subsection{P/S polarization projections}
Let us construct a $\PS$-mode projector $\Pi_{\PS}$, microlocally projects the elastic wave-field $u$ to the compressive ($P$) and the sheer ($S$) wave-fields for a small time-interval, as
\begin{equation*}
\Pi_{\PS} u := \sum_{\pm, l, k} \int_{\R^3} e^{i\phi^{\pm}_{\PS}(t,x,\xi)} a^{\cdot,l}_{\pm,k,\PS}(t,x,\xi) (\hat{\phi}_{k})_l(\xi)\,d\xi,\quad \mbox{where }l=1,2,3,\quad k=1,2.
\end{equation*}

Observe that $p(t,x,\tau,\xi)$ has eigenvalues $\rho\left(\tau^2-c_P^2|\xi|^2\right)$ and $\rho\left(\tau^2-c_S^2|\xi|^2\right)$ with multiplicity $1$ and $2$ respectively (see \eqref{princ_symb}). The matrix $p$ can be diagonalised and there exists unitary matrix $V(t,x,\tau,\xi)$ such that
\begin{equation*}
V p(t,x,\tau,\xi) V^{-1} = \rho\begin{pmatrix}\tau^2-c_P^2|\xi|^2 &0 &0\\0 &\tau^2-c_S^2|\xi|^2 &0\\0 &0 &\tau^2-c_S^2|\xi|^2\end{pmatrix} = D(t,x,\tau,\xi).
\end{equation*}
Let us now consider the symbol
\begin{equation}\label{Symbol_mode-projector}
\Pi_P(t,x,\tau,\xi) := V^{-1}\bmat{1 &0 &0\\0 &0 &0\\0 &0 &0}V
\qquad \mbox{and}\qquad
\Pi_S(t,x,\tau,\xi) := V^{-1}\bmat{0 &0 &0\\0 &1 &0\\0 &0 &1}V.
\end{equation}
Observe that the symbol of $\Pi_{\PS}$ is homogeneous of order $0$ in $|\xi|$ and thus $\Pi_{\PS}$ represents a $0$-th order pseudodifferential operator.

\subsection{ Transmission conditions}
Let $u_I$ be an incoming wave starting at $\Omega_{t_j}^*$ and travelling towards $\Gamma_j=\Sigma_{t_j}$, for $j=0,1,\dots,m$. At $\Gamma_j$ it hits the interface and breaks into two parts $u_R$ the reflected wave and $u_T$ the transmitted wave.
From now on, we will write the subscript $\bullet=I/R/T$ to denote the incoming, reflected or transmitted quantities.

Let us define the Neumann operator at an interface $\Gamma_j$, for $j=0,1,\dots,m$ as
\begin{equation}\label{Neumann}
\mathcal N_{\bullet} u_{\bullet} = (\lambda \div \otimes \text{I} + 2\mu \hat \nabla)u_{\bullet} \cdot \nu_{\bullet}|_\Gamma,
\end{equation}
where $\nu$ is an outward unit normal vector at $\Gamma_j$.
The elastic transmission conditions on the interface $\Gamma_j$, for $j=0,1,\dots,m$ are given as
\begin{equation}\label{transmission_cond}
\begin{aligned}
u_I + u_R =& u_T\\
\mathcal{N}_{I}u_I + \mathcal{N}_{R}u_R =& \mathcal{N}_{T}u_T.
\end{aligned}
\end{equation}
The system \eqref{transmission_cond} can be microlocally inverted to obtain the reflection and the transmission operators $M_R$ and $M_T$, where $M_R u_I|_{\Gamma_j}=u_R$ and $M_T u_I|_{\Gamma_j} = u_T$. Note that the operators $M_R$, $M_T$, are $\Psi$DOs of order $0$ on $\R\times\Gamma_j$, have been calculated explicitly in \cite{CHKUElastic}. The operators $M_R$, $M_T$ changes from interface to interface, but for the sake of notational simplicity we do not mention the influence of $j=0,1,\dots,m$ on them.

\subsection{Parametrix}
Define the operator $\iota: \Sigma_{t,\pm} \to \Sigma_{t,\pm}$ changes from one boundary to its copy in $\Sigma_{t,\pm}$. Consider the boundary operator $M=M_R+\iota M_T$.
To understand the propagation of the wave-field through this broken medium let us consider
\begin{align}\label{Propagation_wave}
\widetilde{F} :&= \mathbf{J}_{\C\to \mathbf{S}} + \mathbf{J}_{\partial\to \mathbf{S}} \sum_{k=0}^{\infty}  \left(\mathbf{J}_{\partial\to\partial}M\right)^k\mathbf{J}_{\C\to\mathbf{S}}
\\
\tilde R_{2T} &= r_{2T} \circ \tilde F,
\end{align}
where $r_{2T}$ is restriction to $t=2T$. Again omitting the proof, it can be shown that $\tilde F \equiv F$ and $\tilde R_{2T} \equiv R_{2T}$ away from glancing rays. In the elastic case it means that $\tilde R_{2T}h_0 \equiv R_{2T}h_0$ for initial Cauchy data $h_0$ such that every broken bicharacteristic originating in $\text{WF}(h_0)$ is disjoint from the both the $P$ and $S$ glancing sets described in \cite{SUV2019transmission}. Recalling that $M = M_R + M_T$, we may write $\tilde R_{2T}$ as a sum of graph FIO indexed by sequences of reflections and transmissions:
\begin{align}
&\tilde R_{2T} = \sum_{ s \in \{R,T\}^k, \lambda \in \{P, S\}^{k+1}} \tilde R_{s,\lambda},
\qquad \qquad
\tilde R_{()} = r_{2T}J\CtoS \\
&\tilde R_{(s_1,\dots,s_k; \lambda_0, \dots, \lambda_k)} = r_{2T}J\BtoS \Pi_{\lambda_k}M_{s_k}J\BtoB \cdots \Pi_{\lambda_2}M_{s_2}J\BtoB \Pi_{\lambda_1}M_{s_1}J\CtoB \Pi_{\lambda_0}.
\end{align}
The solution operator $\tilde F$ likewise decomposes into analogous components $\tilde F_a$.

\bibliographystyle{plain}
 \bibliography{ScatteringControl}
\end{document}